\documentclass[12pt]{amsart}
\usepackage{amsmath,amssymb,amsfonts,amsthm,amsopn}
\usepackage{latexsym,graphicx}
\usepackage{pb-diagram}

  \def\cL{\mathcal{L}}

\newcommand{\tfa}{time-frequency analysis}

\newcommand{\stft}{short-time Fourier transform}

\newcommand{\fif}{if and only if}

\newcommand{\psdo}{pseudodifferential operator}

\newtheorem{theorem}{Theorem}[section]
\newtheorem{lemma}[theorem]{Lemma}
\newtheorem{corollary}[theorem]{Corollary}
\newtheorem{proposition}[theorem]{Proposition}
\newtheorem{definition}[theorem]{Definition}

\newtheorem{remark}[theorem]{Remark}

\newcommand{\beqa}{\begin{eqnarray*}}
\newcommand{\eeqa}{\end{eqnarray*}}

\newcommand{\field}[1]{\mathbb{#1}}
\newcommand{\bR}{\field{R}}        
\newcommand{\bN}{\field{N}}        
\newcommand{\bZ}{\field{Z}}        
\newcommand{\fiola}{FIO(\Xi ,s)}

\def\la{\lambda}

\def\eps{\epsilon}

\def\cS{\mathcal{S}}

\def\cM{\mathcal{M}}

\def\cC{\mathcal{C}}

\def\a{\aleph}

\def\rd{\bR^d}

\def\rdd{{\bR^{2d}}}

\def\lrd{L^2(\rd)}

\def\intrd{\int_{\rd}}
\def\intrdd{\int_{\rdd}}

\def\R{\right)}

\def\<{\left<}
\def\>{\right>}

\def\mv1{M_v^1}

\def\phas{(x,\o )}
\def\mn{(m,n)}
\def\mn'{(m',n')}

\def\wpr{WF^{p,r}_G}

\def\o{\eta}
\def\a{\alpha}
\def\b{\beta}

\def\R{\mathbb{R}}
\def\Ren{\mathbb{R}^d}
\def\Renn{\mathbb{R}^{2d}}

\def\sch{\mathcal{S}}

\def\Fur{\mathcal{F}}

\def\f{\varphi}

\def\Sn2{S_{2}(L^{2}(\Ren))}
\def\S1{S_{1}(L^{2}(\Ren))}
\def\sig00{\sigma_{0,0}}

\def\la{\langle}
\def\ra{\rangle}

\begin{document}
\begin{abstract} We consider Schr\"odinger equations with real-valued smooth Hamiltonians, and non-smooth bounded pseudo-differential potentials,
whose symbols may be not even differentiable. The well-posedness of the Cauchy problem is proved in the frame of the modulation spaces, and results of micro-local propagation of singularities are given in terms of Gabor wave front sets.
\end{abstract}

\title[Propagation of the Gabor Wave Front Set for  Schr\"odinger Equations]{Propagation of the Gabor Wave Front Set for  Schr\"odinger Equations with non-smooth potentials}

\author{Elena Cordero}
\address{Universit\`a di Torino, Dipartimento di Matematica, via Carlo Alberto 10, 10123 Torino, Italy}
\email{elena.cordero@unito.it}
\author{Fabio Nicola}
\address{Dipartimento di Scienze Matematiche,
Politecnico di Torino, corso Duca degli Abruzzi 24, 10129 Torino,
Italy}
\email{fabio.nicola@polito.it}
\author{Luigi Rodino}
\address{Universit\`a di Torino, Dipartimento di Matematica, via Carlo Alberto 10, 10123 Torino, Italy}
\email{luigi.rodino@unito.it}

\subjclass{Primary 35S30; Secondary 47G30}

\subjclass[2010]{35A18, 35A21, 35B65, 35S30, 42C15, 
47G30, 47D08}
\keywords{Schr\"odinger propagator, modulation spaces,
short-time Fourier
 transform, wave front set}
\maketitle
\section{Introduction}

The authors in \cite{fio3} and in collaboration with Gr\"ochenig in \cite{Wiener} proposed a new approach to the calculus of
the Fourier integral operators (FIOs) in terms of time-frequency localization, cf.\ \cite{Daubechies90} and \cite{book}, also named Gabor
analysis. The FIOs under consideration were of the type of those appearing in the study of the Schr\"odinger
  equations, typically  the phase function being a homogeneous function of degree 2 in the whole of the phase space variables.
  With respect to the standard representations of FIOs, the time-frequency representation looks more involved, since old and new phase-space variables appear simultaneously, and everything depends on the choice of the so-called window function. On the other hand, the problem of the caustics is automatically solved in this new setting, see \cite{Wiener}, and the expression provides an excellent tool for the numerical analysis, see \cite{fio3}.\par
  In the present paper we apply the aforesaid results to the
  analysis of the Schr\"odin\-ger
  equation. With respect to the enormous existing literature, our
  results will be new in the following aspects. Fixed a
  real-valued Hamiltonian, homogeneous of degree $2$, we allow a
  pseudo-differential perturbation (called also potential in the
  following) with a bounded, complex-valued, non-smooth symbol,
  for which even differentiability may be lost. A global-in-time
  propagator is constructed in the class of the FIOs in
  \cite{Wiener}, and well-posedness of the Cauchy problem is
  deduced in suitable modulation spaces. About propagation of
  singularities, which is our main concern in this paper, the
  known results do not apply to such situation. We are then led to
  a new definition of Gabor wave front set, which allows the
  expression of optimal results of propagation  in our
  context.\par
  Let us be more precise. The aim of the paper is to study the
  representation in terms of time-frequency analysis of the
  propagator $e^{i t H}$,
  \begin{equation}\label{PropH}
  H=a(x,D)+\sigma(x,D),
  \end{equation}
 providing the solution to the Cauchy problem: 
 \begin{equation}\label{C1intro}
\begin{cases} i \displaystyle\frac{\partial
u}{\partial t} +a(x,D)u+\sigma(x,D) u=0\\
u(0,x)=u_0(x).
\end{cases}
\end{equation}
The Hamiltonian  $a(x,D)$ is a pseudodifferential operator in the Kohn-Nirenberg form
\begin{equation}\label{kohn}
A f(x)=a(x,D) f(x)=\intrd e^{2\pi i \la x,\xi\ra }a(x,\xi)\hat{f}(\xi)\,d\xi,
\end{equation}
 where the symbol $a(z)$,
$z=(x,\xi)$, is real-valued positively homogeneous of degree 2,
i.e.\ $a(\lambda z)=\lambda^2 a(z)$ for $\lambda>0$, with
$a\in\cC^\infty (\rdd\setminus{0})$. This implies $a(x,D)$ is formally self-adjoint modulo $0$-order perturbations. Basic examples are
real-valued quadratic forms $a(z)$, including the cases when $i
\partial_t +a(x,D)$ is the free particle or the
harmonic oscillator operator. When $a(z)$ is not a polynomial, we
shall assume $a(z)$ modified in a bounded neighborhood of the
origin, in such a way that we have $a\in\cC^\infty(\rdd)$ keeping
real values. As we shall see, cf.\ Example 4 below, the
singularity at the origin of $a(z)$ can be admitted as well, by
absorbing it in a non-smooth potential. The pseudodifferential
operator $a(x,D)$ enters the classes of \cite{Shubin91}, see
also \cite{helffer84}, to which we address for the symbolic
calculus and other properties, see also the next Section \ref{2}.

Concerning the potential $\sigma(x,D)$, the regularity assumptions will be expressed in terms of the modulation spaces, introduced by Feichtinger in \cite{F1}, see also \cite{fg89jfa},  and in the last decades applied  in many fields of mathematics, in particular in PDEs.
 We need first to recall some basic notations. The time-frequency
shifts (phase-space shifts) are denoted by
\begin{equation}\label{TF-S}
\pi(z) f(t)= M_\eta T_x f(t)= e^{2\pi i \la t, \eta\ra}f(t-x),
\quad z=(x,\eta).
\end{equation}
The short-time Fourier transform (STFT) of a function or
distribution $f$ on $\rd$ with respect to a Schwartz window
function $g\in\cS(\rd)\setminus\{0\}$ is defined by
\begin{equation}\label{STFT}
V_g f(x,\eta)=\la f ,\pi(z)g\ra =\intrd f(v) \overline{g(v-x)}
e^{- 2\pi i \la \eta, v\ra}\, dv,\quad z=\phas \in\rdd.
\end{equation}
Assuming for simplicity $\|g\|_2=1$, from $V_g f$ we may
reconstruct $f$ by the formula
\begin{equation}\label{Invform}
f=\intrdd V_g f\phas M_\eta T_x g\, dx d\eta
\end{equation}
(see the the next Subsection \ref{2.1} for details). 

 Fix a not null window function
$\psi\in\cS(\rdd)$ and perform the STFT $V_\psi \sigma(z,\zeta)$
of $\sigma(x,\xi)$ with respect to $z=(x,\xi)\in\rdd$ with dual
variables $\zeta\in\rdd$.
\begin{definition}\label{Ssw}
We say that $\sigma\in\cS'(\rdd)$ belongs to the class $S^s_w$,
$s\geq0$, if
\begin{equation}\label{Ssweq}
|V_\psi \sigma(z,\zeta)|\leq C \la \zeta\ra^{-s},\quad z,\zeta\in
\rdd,
\end{equation}
for a suitable $C>0$ independent of $z$ and $\zeta$, with $\la \zeta\ra=(1+|\zeta|^2)^{1/2}$.
\end{definition}

Our assumption on the potential  will be $\sigma \in S^s_w$ with $s>2d$. Observe that
\begin{equation}\label{inteSs}
\bigcap_{s\geq 0}S^s_w=S^0_{0,0},
\end{equation}
where $S^0_{0,0}$ is the class of all $\sigma\in\cC^\infty(\rdd)$ satisfying 
\begin{equation}\label{HCS0}
|\partial^\a \sigma(z)|\leq C_\a, \quad \a\in\bZ^{2d}_+,
\,\,z=(x,\xi)\in\rdd.
\end{equation}
Whereas, for $s\to 2d+$, the symbols in $S^s_w$ have a smaller
regularity. More precisely, if $s>2d+m$, then  $S^s_w\subset \cC^m(\rdd)$. In particular, for
$s>2d$, $S^s_w\subset \cC^0(\rdd)$,  but the differentiability is lost
in general as soon as $s\leq 2d +1$. \par It is worth to mention
now the definition of the Sj\"ostrand class $S_w$, see
\cite{wiener30}, \cite{wiener31} and \cite{charly06}, given by all
the symbols $\sigma$ for which
\begin{equation}\label{Sjostr}
\intrd \sup_{z\in\rdd} | V_\psi \sigma(z,\zeta)|\,d\zeta <\infty.
\end{equation}
Note that
$$ \bigcup_{s>2d} S^s_w\subset S_w\subset \cC^0(\rdd).
$$
In the present paper we shall not treat the case $\sigma\in S_w$, let
us refer to \cite{CGNR13} where quadratic Hamiltonians with a
Sj\"ostrand  potential are studied.\par
Given any
linear continuous operator $P:\cS(\rd)\to \cS'(\rd)$, its
time-frequency representation is provided by the (continuous)
Gabor matrix
\begin{equation}\label{GbM}
k(w,z):=\la P\pi(w)g,\pi(z) g\ra,\quad w,z\in\rdd
\end{equation}
so that
\begin{equation}\label{KGbM}
V_g(Pf)(z)=\intrdd k(w,z) V_g f(w) \,dw.
\end{equation}
Time-frequency representations give a deep insight into the
properties of relevant classes of operators, see for example \cite{Benyi et
all,locNC09,bertinoro,charly06,Miyachi-NicolaRivetti-Tabacco-Tomita,wh}.
 We want to
study the Gabor matrix $k(t,w,z)$ of the propagator $e^{i t H}$.
Its structure will be linked, as expected, to the Hamiltonian
field of $a(x,\xi)$.  Namely, consider
\begin{equation}\label{1.11}
\begin{cases}
 2\pi\dot{x}=-\nabla  _\xi a ( x,\xi) \\
 2\pi \dot{\xi}=\nabla _x a (x,\xi)\\
 x(0)=y,\ \xi(0)=\eta,
\end{cases}
\end{equation}
(the factor $2\pi$ depends on our normalization of the STFT).
Under our assumptions, the solution
$\chi_t(y,\eta)=(x(t,y,\eta),\xi(t,y,\eta))$ exists for all
$t\in\bR$ and defines a symplectic diffeomorphism
$\chi_t:\,\bR_{y,\eta}^{2d}\to\bR_{x,\xi}^{2d}$ homogeneous of
degree $1$ with respect to $w=(y,\eta)$ for large $|w|$, for every
fixed $t\in\bR$.
\begin{theorem}\label{T1.1}
Let the preceding assumptions be satisfied, in particular let $\sigma\in S^s_w$, $s>2d$,  and let $k(t,w,z)$ be
the Gabor matrix of the Schr\"{o}dinger propagator $e^{i t H}$.
Then there exists $C=C(t,s)>0$ such that 
\begin{equation}\label{KT1.1}
|k(t,w,z)|\leq C\la z-\chi_t(w)\ra^{-s},\quad
z=(x,\xi),\,w=(y,\eta)\in\rdd.
\end{equation}
\end{theorem}
According to the notations of \cite{Wiener}, this can be rephrased
as $e^{it H}\in FIO(\chi_t,s)$. For $t$ sufficiently small
our assumptions yield $\det \frac{\partial x}{\partial
y}(t,y,\eta)\not=0$ in the expression of $\chi_t$, and
\eqref{KT1.1} is then equivalent to
\begin{equation}\label{1.13}
(e^{i t H}u_0)(t,x)=\intrd e^{2\pi i \Phi(t,x,\eta)}
b(t,x,\eta)\widehat{u}_0(\eta)\,d\eta,
\end{equation}
with the phase $\Phi$ linked to $\chi_t$ as standard and
$b(t,\cdot)\in S^s_w$, see \cite[Theorem 4.3]{Wiener}. In the
classical approach, cf.\ \cite{wiener1}, the occurrence of caustics
makes the validity of \eqref{1.13} local in time. So for $t\in
\bR$ one is led to multiple compositions of local representations,
with unbounded number of variables possibly appearing in the
expression. Whereas $k(t,w,z)$ obviously keeps life for every
$t\in\bR$, and the estimates \eqref{KT1.1} hold for $\chi_t$ with
$t\in\bR$.

Under the assumption
$\sigma\in S^s_w$, $s> 2d$, natural
functional frame to express boundedness and propagation results  for $e^{it H}$ is given by the modulation spaces (the classes $S^s_w$ are
special cases), see \cite{F1} and the short survey in Section
\ref{2}.\par We begin to recall here that for  $1\leq p\leq \infty$, $r\in\bR$, the modulation space $M^p_r(\rd)$
is  defined as the space of all $f\in\cS'(\rd)$ for which
\begin{equation}\label{minfty}
\|f\|^p_{M^p_r(\rd)}=\intrdd |V_g f(z)|^p \la z\ra^{pr} dz<\infty
\end{equation}
(with obvious modifications for $p=\infty$).
Let us now define the  Gabor wave front set $\wpr(f)$ under our consideration.
\begin{definition}\label{defWFgr}
Let  $g\in\cS(\rd)$, $g\not=0$, $r>0$. For $f\in M^p_{-r}(\rd)$, $z_0\in \rdd$, $z_0\not=0$, we say that $z_0\notin \wpr (f)$ if there exists an open conic neighborhood $\Gamma_{z_0}\subset \rdd$ containing $z_0$ such that for a suitable constant $C>0$
\begin{equation}\label{5.1}
\int_{\Gamma_{z_0}}|V_gf(z)|^p \la z\ra^{pr}<\infty
\end{equation}
(with obvious changes for $p=\infty$).
\end{definition}
Then $\wpr (f)$ is well-defined as conic closed subset of $\rdd\setminus\{0\}$.
Our main results are  summarized as follows.
\begin{theorem}\label{T1.5}
Consider $\sigma\in S^s_w$, $s>2d$, $1\leq p\leq\infty$. Then
\begin{equation}\label{T1.5eq1}
e^{i t H}: M^p_r(\rd)\to M^p_r (\rd)
\end{equation}
continuously, for $|r|< s-2d$. Moreover,  for $u_0\in M^p_{-r}(\rd)$,
\begin{equation}\label{T1.5eq2}
WF^{p,r}_G(e^{i t H} u_0)=\chi_t(WF^{p,r}_G (u_0)),
\end{equation}
provided $0<2r<s-2d$.
\end{theorem}
Observe the more restrictive assumption  on $r$ for \eqref{T1.5eq2}, with respect to that for
\eqref{T1.5eq1}. 
\par
As an elementary example consider the perturbed harmonic oscillator (studied in Example 4 in the sequel)
\begin{equation}\label{5.31eq}
\begin{cases} i \partial_t
 u -\frac{1}{4\pi}\partial^2_x u+\pi x^2 u+ |\sin x|^\mu u=0\\
u(0,x)=u_0(x)
\end{cases}
\end{equation}
with $\mu>1$. We shall prove that $|\sin x|^\mu\in S^{\mu+1}_w$ and from Theorem \ref{T1.5} we have that the Cauchy problem is well-posed for $u_0\in M^p_{r}(\R)$, $|r|<\mu-2$ and the propagation of $\wpr\, (u(t,\cdot))$ for $t\in \bR$ takes place as in Theorem \ref{T1.5} for $0<r<\mu/2-1$, where 
\begin{equation}\label{5.20eq}
\chi_t(y,\eta)=\begin{pmatrix}(\cos
t)I&(-\sin t)I\\(\sin t)I&(\cos
t)I\end{pmatrix} \begin{pmatrix}y\\\eta\end{pmatrix}
\end{equation}
with $I$ being the identity matrix.
\par
Using \eqref{inteSs} and \eqref{T1.5eq2}, we may recapture the known results for the propagation in the case of a smooth potential, i.e. $\sigma\in S^0_{0,0}$.
We  define  the wave front set $WF_G(f)$ by stating  $z_0\notin WF_G (f)$ if there exists an
open conic set $\Gamma_{z_0}\subset \rdd$ containing $z_0$ such
that for every $r>0$
\begin{equation}\label{WFSeq}
|V_g f(z)|\leq C_r \la z\ra^{-r},\quad z\in \Gamma_{z_0}
\end{equation}
for a suitable $C_r>0$.
Then the estimate \eqref{KT1.1} is satisfied for every $s$ and from Theorem \ref{T1.5} we recapture for $u_0\in\cS'(\rd)$
\begin{equation}\label{WT1.3}
WF_G(e^{i t H} u_0)=\chi_t(WF_G (u_0)).
\end{equation}
This identity is contained in preceding results. Although it is impossible to do justice to the vast literature in this
connection, let us mention some of the related contributions.
The pioneering work is that of H\"{o}rmander \cite{hormanderglobalwfs91} 1991, who defined the  wave front set in \eqref{WFSeq} as well as its analytic version,  and proved  \eqref{WT1.3} in the case of the metaplectic operators (cf. \cite{folland89}).
For subsequent results providing \eqref{WT1.3} and its analytic-Gevrey version for general smooth symbols, let us refer to \cite{Hassel-Wunsch,ito,ito-nakamura,Martinez,Martinez2,Mizuhara,Nakamura,Nakamura2,wunsch}. The wave front sets introduced there under different names actually coincide with those of  H\"{o}rmander 1991, cf. \cite{RWwavefrontset}, \cite{sw} and  \cite{Cappiello-shulz}. Still concerning propagation of singularities in the case of  smooth or analytic symbols we refer to \cite{Craig-Kappler,Melrose,Robbiano-Zuily,Robbiano-Zuily2,Weinstein}. Besides, concerning global-in-time representations of $e^{it H}$, solving the problem of the caustics for smooth symbols, see \cite{wiener1,wiener3,wiener4,gz,tataru}.

Despite the abundance of contributions in the case when Hamiltonians and potentials are smooth, our study of propagation of singularities in the case of non-smooth potentials is new in literature, as far as we know. We hope, in future papers, to extend the analysis to non-smooth Hamiltonians as well, with applications to propagations for non linear Schr\"odinger equations. In such order of ideas,  time-frequency methods represent an important tool. Beside \cite{Wiener,fio3} see \cite{Benyi et all,locNC09,CGNR13,cn,fio5,fio1,bertinoro,CNEdcds,kki1,kki2,kki4,Miyachi-NicolaRivetti-Tabacco-Tomita,nicola,wh}.\par
The contents of the next sections are the following. In Section \ref{2}, after a survey on modulation spaces, Shubin classes and construction of propagators in their setting, we provide some improvements of the calculus in \cite{Wiener}
 for the classes $FIO(\chi,s)$, as preparation for the sequel. In Section $3$ we treat the unperturbed equation, giving a global construction of the propagator in terms of time-frequency analysis. In Section $4$ we add the non-smooth bounded perturbation, and we prove the main results of representation and continuity, stated before. The propagation result is proved in Section $5$, where we also give some examples.

\vskip0.3truecm \textbf{Notation.}
The Schwartz class is denoted by
$\sch(\Ren)$, the space of tempered
distributions by  $\sch'(\Ren)$.
The brackets $\la\cdot,\cdot \ra$  denote either the inner product on $\rd$ or  the extension to $\sch '
(\Ren)\times\sch (\Ren)$ of the inner
product $\la f,g\ra=\int f(t){\overline
{g(t)}}dt$ on $L^2(\Ren)$. The Fourier
transform is normalized to be ${\hat
  {f}}(\o)=\Fur f(\o)=\int
f(t)e^{-2\pi i\la  t,\o\ra}dt$.


We
shall use the notation
$A\lesssim B$ to express the inequality
$A\leq c B$ for a suitable
constant $c>0$, and  $A
\asymp B$  for the equivalence  $c^{-1}B\leq
A\leq c B$.

\vskip0.3truecm
\section{Preliminaries}\label{2}
  We recall the basic
concepts  of \tfa\ and  refer the  reader to \cite{book} for the full
details.
\subsection{The Short-time Fourier Transform }\label{2.1}
Consider a distribution $f\in\cS '(\rd)$ and a Schwartz function
$g\in\cS(\rd)\setminus\{0\}$ (the so-called {\it window}). The
short-time Fourier transform (STFT) of $f$ with respect to $g$ is defined in \eqref{STFT}.
 The  \stft\ is well-defined whenever  the bracket $\langle \cdot , \cdot \rangle$ makes sense for
dual pairs of function or distribution spaces, in particular for $f\in
\cS ' (\rd )$ and $g\in \cS (\rd )$, $f,g\in\lrd$. If $f,g\in\cS(\rd)$, then $V_gf\in\cS(\rdd)$.
\par
 We recall the following pointwise inequality of the \stft\
 \cite[Lemma 11.3.3]{book}, useful when one needs to change window functions.
 \begin{lemma}\label{changewind}
 If  $g_0,g_1,\gamma\in\cS(\rd)$ such
 that $\la \gamma, g_1\ra\not=0$ and
 $f\in\cS'(\rd)$,  then the inequality
 $$|V_{g_0} f(x,\xi)|\leq\frac1{|\la\gamma,g_1\ra|}(|V_{g_1} f|\ast|V_{g_0}\gamma|)(x,\xi)$$
 holds pointwise for all $(x,\xi)\in\rdd$.
 \end{lemma}


\subsection{Modulation spaces and Shubin classes}\label{2.2}
Weighted modulation spaces measure the decay of the STFT on the time-frequency (phase space) plane and were introduced by Feichtinger in the 80's \cite{F1}.

\emph{Weight Functions.}  A weight function $v$ is submultiplicative if $ v(z_1+z_2)\leq v(z_1)v(z_2)$, for all $z_1,z_2\in\Renn.$  We consider the weight functions
\begin{equation}\label{weight} v_s(z)=\la z\ra^s=(1+|z|^2)^{\frac s 2},\quad s\in\R,
\end{equation}
which are submultiplicative for $s\geq0$.\par
 For $s\geq0$, we denote by $\mathcal{M}_{v_s}(\rdd)$ the space of $v_s$-moderate weights on $\rdd$; these  are measurable positive functions $m$ satisfying $m(z+\zeta)\leq C
v_s(z)m(\zeta)$ for every $z,\zeta\in\rdd$.

\begin{definition}  \label{prva}
Given  $g\in\cS(\rd)$, $s\geq0$, a  weight
function $m\in\mathcal{M}_{v_s}(\rdd)$, and $1\leq p,q\leq
\infty$, the {\it
  modulation space} $M^{p,q}_m(\Ren)$ consists of all tempered
distributions $f\in \cS' (\rd) $ such that $V_gf\in L^{p,q}_m(\Renn )$
(weighted mixed-norm spaces). The norm on $M^{p,q}_m(\rd)$ is
\begin{equation}\label{defmod}
\|f\|_{M^{p,q}_m}=\|V_gf\|_{L^{p,q}_m}=\left(\int_{\Ren}
  \left(\int_{\Ren}|V_gf(x,\xi)|^pm(x,\xi)^p\,
    dx\right)^{q/p}d\xi\right)^{1/q}  \,
\end{equation}
(obvious changes if $p=\infty$ or $q=\infty$).
\end{definition}
 When $p=q$, we simply write $M^{p}_m(\rd)$ instead of $M^{p,p}_m(\rd)$. The spaces $M^{p,q}_m(\rd)$ are Banach spaces and every nonzero $g\in M^{1}_{v_s}(\rd)$ yields an equivalent norm in \eqref{defmod} and so $M^{p,q}_m(\Ren)$ is independent on the choice of $g\in  M^{1}_{v_s}(\rd)$.

In particular,  we recover the H\"ormander
class
\begin{equation}\label{HC}S^0_{0,0}=\bigcap_{s\geq 0}M^{\infty}_{1\otimes v_s}(\rdd).\end{equation}
Note that, for any $1\leq p,q\leq\infty$,

\begin{equation}\label{HCbis}\bigcap_{s\geq 0}M^{p,q}_{ v_s}(\rd)=\cS(\rd),\quad \bigcup_{s\geq 0}M^{p,q}_{ v_{-s}}(\rd)=\cS'(\rd).\end{equation}
In the introduction we used the short notations $M^\infty_r(\rd)$ for $M^\infty_{v_r}(\rd)$ and $S^s_r$  for $M^\infty_{1\otimes v_s}(\rdd)$.
Fix $g\in\cS(\rd)\setminus\{0\}$. The adjoint operator of $V_g$,  defined by $\la V_g^\ast F, h\ra=\la F,V_g h\ra$, can be written as
 \begin{equation}\label{adj}V_g^\ast F=\intrdd F(x,\xi) \pi(x,\xi) g dx d\xi,
\end{equation}
 $ V_g^\ast$ maps  the Banach space $L^{p,q}_m(\rdd)$ into $M^{p,q}_m(\rd)$, in particular it maps $\cS(\rdd)$ into $\cS(\rd)$ and the same for their dual spaces. In particular, if $F=V_g f$ we obtain the  inversion formula for the STFT
 \begin{equation}\label{treduetre}
 {\rm Id}_{M^{p,q}_m}=\frac 1 {\|g\|_2^2} V_g^\ast V_g
 \end{equation}
and the same holds when replacing $M^{p,q}_m(\rd)$ by $\cS(\rd)$ or $\cS'(\rd)$.\par
In the subsequent Section $5$ we shall use the following properties.
\begin{lemma}\label{fabio} Consider $\mu>0$. Then the function $f(x)=|\sin\,x |^\mu\in M^{\infty}_{1\otimes v_{\mu+1}}(\bR)$.
\end{lemma}
\begin{proof} Consider a  window function $g\in\cC^\infty_0(\bR)$,  with supp\,$g\subset [-\pi/4,\pi/4]$ to compute the STFT $V_g f$ with
$f(x)=|\sin\,x |^\mu$. Then  $|V_g f(x,\xi)|$ is a periodic function of period $\pi$ in the $x$ variable. So
$$\| f\|_{M^\infty_{1\otimes \mu+1}}=\sup_{|x|\leq \pi/2} \sup_{\xi\in\bR}\la \xi\ra^{\mu+1} |V_g f|(x,\xi).$$
Now observe that supp $T_x g\subset [-3\pi/4,3\pi/4]$, for $x\in [-\pi/2,\pi/2]$, and on that interval $f(x)=|x|^\mu \f(x)$, with $\f\in\cC^\infty_0(\bR)$.
We can write,
$$V_g f(x,\xi)=\int_0^{+\infty} e^{-2\pi i t \xi} t^\mu\f(t) \overline{g(t-x)}\,dt+\int_{-\infty}^0 e^{-2\pi i t \xi} (-t)^\mu\f(t) \overline{g(t-x)}\,dt:=A+B.$$
So it suffices to estimate the integral $A$, the estimate of $B$ is analogous. Setting $F_x(t)=e^t\f(t) g(t-x)\in\cS(\bR)$ we observe that the family $\{F_x\}_{x\in[-\pi/2,\pi/2]}$ belongs to a bounded subset of $\cS(\bR)$. Now
\begin{align*} A&= \int_0^{+\infty} e^{-2\pi i t \xi} t^\mu e^{-t} e^t\f(t) g(t-x)\,dt\\
&=\frac{\Gamma(\mu+1)}{(1+2\pi i \xi)^{\mu+1}}\ast \widehat{F_x}(\xi),
\end{align*}
and this yields
$$\la \xi\ra^{\mu+1}|A|\lesssim \frac{\la \xi\ra^{\mu+1}}{(1+2\pi i \xi)^{\mu+1}}\ast (\la \xi\ra^{\mu+1}\widehat{F_x}(\xi))\in L^\infty(\bR)
$$
by Young's inequality, since the first factor of the convolution product is bounded and the second one lies in a bounded subset of $\cS(\bR)\subset L^1(\bR)$.
\end{proof}
\begin{corollary}\label{fabio2} Consider the symbol $\sigma(x,\xi)=|\sin\,x |^\mu$ on $\bR^2$. Then we have $\sigma\in M^{\infty}_{1\otimes v_{\mu+1}}(\bR^2)$.
\end{corollary}
\begin{proof}  It is an immediate consequence of Lemma \ref{fabio}. Indeed, taking $\psi(x,\xi)=g(x)\f(\xi)$, with $g$ being the $1$-dimensional window of the previous proof and $\f\in\cS(\bR)$, we have $V_\psi \sigma((x_1,x_2),(\xi_1,\xi_2))= V_g (|\sin(\cdot) |^\mu) (x_1,\xi_1)V_\f 1(x_2,\xi_2)$
and the thesis follows since $\la (\xi_1,\xi_2)\ra\leq \la \xi_1\ra \la \xi_2\ra$ and  $1\in S^0_{0,0}\subset M^\infty_{1\otimes v_s}(\bR)$, for every $s\geq0$.
\end{proof}
\begin{proposition}\label{fabio3}
Let $h\in\cC^\infty(\rd\setminus\{0\})$ be positively homogeneous of degree $r>0$, i.e.\ $h(\lambda x)=\lambda^r h(x)$ for $x\not=0$, $\lambda>0$, and $\chi\in C^\infty_0(\rd)$. Set $f=h\chi$. Then, for $\psi\in\cS(\rd)$ there exists a constant $C>0$ such that
\[
|V_\psi f(x,\xi)|\leq C(1+|\xi|)^{-r-d},\quad x,\xi\in\rd.
\]
 \end{proposition}
 \begin{proof}
 We know that the Fourier transform of $h$ is a homogeneous distribution of degree $-r-d$, smooth in $\rd\setminus\{0\}$ \cite[Vol.1, Theorems 7.1.16, 7.1.18]{hormander3}. Hence, if $\chi'\in C^\infty_0(\rd)$, $\chi=1$ in a neighborhood of the origin we have 
 \begin{equation}\label{equl}
 |(1-\chi'(\xi)) \widehat{h}(\xi)|\leq C(1+|\xi|)^{-r-d},\quad \xi\in\rd.
 \end{equation}
 On the other hand, by the very definition of the STFT we have 
 \[
 |V_\psi f(x,\xi)|\leq|\big((\chi'\widehat{h})\ast_{\xi}\widehat{\chi}\big)|\ast_
 \xi|\widehat{\overline{\psi}}|+ |\big((1-\chi')\widehat{h}\big)\ast_{\xi}\widehat{\chi}|\ast_{\xi}|\widehat{\overline{\psi}}|.
 \]
Since $\widehat{\overline{\psi}},\widehat{\chi}\in\cS(\rd)$, the first term in the right-hand side has a rapid decay, bacause $\mathcal{E}'\ast\cS\subset\cS$,  whereas the second term is estimated using \eqref{equl}, as at the end of the proof of Lemma \ref{fabio}.

  \end{proof}
\par

Here we are interested in operators with symbols in the Shubin classes (cf.\ \cite{Shubin91}, Helffer \cite{helffer84}); indeed,  we shall use them as symbol and phase spaces for the unperturbed initial value problem for Schr\"odinger equations.
\begin{definition}\label{shubinclass}
For $m\in\bR$, the class $\Gamma^m(\rdd)$ is the set of functions $a\in\mathcal{C}^\infty(\rdd)$ such that for every $\a \in \bZ^{2d}_+$ there exists a constant $C_\a >0$ such that:
$$|\partial^\a_z a(z)|\leq C_\a v_{m-|\a |}(z),\quad z\in\rdd,
$$
where we recall $v(z)=\la z\ra$ is defined in \eqref{weight}
\end{definition}

Consider $a_j\in\Gamma^{m_j}(\rdd)$ with $m_j$ being a decreasing sequence tending to $-\infty$. Then a function $a\in\cC^\infty(\rdd)$ satisfies
\begin{equation}\label{espansion}
a\sim \sum_{j=1}^{\infty} a_j
\end{equation}
if
$$\forall r\geq 2\quad a-\sum_{j=1}^{r-1}a_j\in \Gamma^{m_r}(\rdd).
$$

Namely, our symbol class well be a subclass of $\Gamma^m(\rdd)$, defined as follows \cite[Sec. 1.5 ]{helffer84}.
\begin{definition}
A function $a$ is in the class $\Gamma^{m,cl}(\rdd)$ if $a \in\Gamma^m(\rdd)$ and admits an asymptotic expansion
\begin{equation}\label{espansiono}
a\sim \sum_{j=0}^{\infty} a_{m-j},
\end{equation}
where $a_{m-j}\in\cC^\infty(\rdd)$  and satisfies $ a_{m-j}(\lambda z)=\lambda^{m-j}a_{m-j}(z)$, for $|z|\geq 1$ and $\lambda\geq 1$. The function $a_m$ corresponding to $j=0$ in the expansion \eqref{espansiono} is called \emph{principal symbol} of the symbol $a$.
\end{definition}
For $a\in \Gamma^m(\rdd)$, the corresponding pseudodifferential operator $a(x,D)$ is defined by \eqref{kohn}. 
\begin{definition}We say that $A\in G^{m}$ (resp.\ $A\in G^{m,cl}$) if its symbol satisfies  $a\in \Gamma^{m}(\rdd)$ (resp.\ $a\in \Gamma^{m,cl}(\rdd)$).
\end{definition}

A pseudodifferential operator $A\in  G^{m,cl}$ is called globally elliptic if there exist $R>0$, $C>0$ such that
\begin{equation}\label{ellitticsimb}
|a_m(z)|\geq C \la z\ra^m,\quad \mbox{for}\quad z\in\rdd,\,\,|z|\geq R,
\end{equation}
where $a_m$ is the principal symbol.

\subsection{Phase functions and canonical transformations}\label{2.3}
Let $a\in \Gamma^{2,cl}(\rdd)$ with real principal symbol $a_2$. The related classical evolution, given by the linear Hamilton-Jacobi
system, following our normalization can be written as
\begin{equation}\label{HS}
\begin{cases}
 2\pi \partial_t x(t,y,\eta)=-\nabla  _\xi a_2 ( x(t,y,\eta),\xi(t,y,\eta)) \\
 2\pi \partial_t\xi(t,y,\eta)=\nabla _x a_2 (x(t,y,\eta),\xi(t,y,\eta))\\
 x(0,y,\eta)=y,\\
 \xi(0,y,\eta)=\eta.
\end{cases}
\end{equation}
%

The solution $(x(t,y,\eta),\xi(t,y,\eta))$ exists for every $t\in\bR$. Indeed,
setting $u:=(x,\xi)$, $F(u):=(-\nabla_\xi a(u), \nabla _x a(u))$,
 the initial value problem \eqref{HS} can be rephrased as
\begin{equation}\label{IVP}
u^\prime (t)=F(u(t)),\quad u(t_0)=u_0,
\end{equation}
in the particular case $t_0=0$.
Observe that $ a\in \Gamma^{2,cl}(\rdd)$ implies $F_j\in  \Gamma^{1,cl}(\rdd)$, for $j=1,\dots,2d$ and $\partial^\a F_j\in  \Gamma^{0,cl}(\rdd)$, for every $|\a|>0$,  $j=1,\dots,2d$, hence in particular $F: \rdd\to \rdd$ is a Lipschitz continuous mapping. Thus the previous ODE  is an autonomous ODE with a mapping $F\in\cC^\infty(\rdd\to\rdd)$ having at most linear growth, hence $\|F(u)\|\lesssim 1+\|u \|$. This implies that for each $u_0\in\rdd$ and $t_0\in\bR$ there exists a unique classical global solution $u\,:\, \bR\to \rdd$ (in this case  $u\in \cC^\infty(\bR\to\rdd)$ since $F\in\cC^\infty(\rdd\to\rdd)$)  to \eqref{IVP}. Moreover the solution maps $S_{t_0}(t)\,:\rdd\to\cC^\infty(\bR\to\rdd)$,  defined by $S_{t_0}(t_0)u_0=u(t)$, and  $S_{t_0}(t_0)={\rm Id}$, the identity operator on $\rdd$, are  Lipschitz continuous mappings, obey the time translation invariance $S_{t_0}(t)=S_0(t-t_0)$ and the group laws
\begin{equation}\label{prodotto}
S_{0}(t)S_0(t')=S_0(t+t'),\quad S_0(0)={\rm Id}.
\end{equation}
Observe  that $S_{0}(t)$ is a bi-Lipschitz diffeomorphism with $S_{0}^{-1}(t)=S_{0}(-t)$.
To be consistent with the notations of the earlier paper \cite{Wiener},  we call the bi-Lipschitz diffeomorphism
\begin{equation}\label{mappachi}\chi_t(y,\eta):=S_0(t)(y,\eta),\quad (y,\eta)\in\rdd.\end{equation}

The theory of Hamilton-Jacobi allows to find a $T>0$ such that for $t\in ]-T,T[$ there exists a phase function  $\Phi(t,x,\eta)$, solution of the eiconal equation (cf.\ \cite[(3.2.12),(3.2.13)]{helffer84})
\begin{equation}\label{eiconal}
\begin{cases}
2\pi \partial_t\Phi+a_2 (x,\nabla_x\Phi)=0\\
\Phi(0,x,\eta)=x \eta
\end{cases}
\end{equation}
The phase $\Phi(t,x,\eta)\in\cC^\infty(]-T,T[,\Gamma^2(\rdd))$ is real-valued since the principal
symbol $a_2(x,\xi)$ is real-valued, moreover $\Phi$ fulfills  the condition
of non-degeneracy:
\begin{equation}\label{detmisto}
|\det \partial_{x,\eta}^2 \Phi(t,x,\eta)|\geq c>0,\quad
(t,x,\eta)\in ]-T,T[\times(\rdd\setminus\{0\}),
\end{equation}
after possibly shrinking $T>0$ (cf.\ \cite[Pages 142-143]{helffer84} and \cite{bertinoro}).

 The relation between the phase $\Phi$ and the canonical transformation $\chi$ is given by
\begin{equation}\label{rel-chi-phi}
(x,\nabla_x\Phi(t,x,\eta))=\chi_t(\nabla_\eta\Phi(t,x,\eta),\eta),\quad t\in ]-T,T[.
\end{equation}
In particular,
\begin{equation}\label{cantra} \left\{
                \begin{array}{l}
                y(t,x,\eta)=\nabla_{\eta}\Phi(t,x,\eta)
                \\
               \xi(t,x,\eta)=\nabla_{x}\Phi(t,x,\eta), \rule{0mm}{0.55cm}
                \end{array}
                \right.
\end{equation}
and there exists $\delta>0$ such that
\begin{equation}\label{detcond2}
   |\det\,\frac{\partial x}{\partial y}(t,y,\eta)|\geq \delta \quad t\in ]-T,T[.
\end{equation}
Observe that each component of $\chi_t$ is a function in $\in\cC^\infty(]-T,T[,\Gamma^1(\rdd))$, positively homogeneous of degree $1$ for $(y,\eta)$ large.  Moreover, using \eqref{prodotto} we observe that the same holds in fact for every $t\in\bR$.

For $t\in ]-T,T[$, the  phase function $\Phi(t,\cdot)$ above is  a \emph{tame} phase,  and similarly for the canonical transformation $\chi_t$, according to the following definition \cite[Definition 2.1]{Wiener}:
\begin{definition}\label{de}
A real and smooth phase function $\Phi(x,\eta)$ on $\rdd$ is called \emph{tame} if:\\
(i) For $z=\phas$,
\begin{equation}\label{phasedecay}
|\partial_z^\a \Phi(z)|\leq C_\a,\quad |\a|\geq 2;\end{equation}
(ii) There exists $c>0$ such that the following condition
of non-degeneracy holds:
\begin{equation}\label{detcond}
   |\det\,\partial^2_{x,\eta} \Phi(x,\o)|\geq c.
\end{equation}
 \par The mapping  defined by $(x,\xi)=\chi(y,\o)$, which solves the system
 \begin{equation}\label{cantra2} \left\{
                 \begin{array}{l}
                 y(x,\eta)=\nabla_{\eta}\Phi(x,\eta)
                 \\
                \xi(x,\eta)=\nabla_{x}\Phi(x,\eta), \rule{0mm}{0.55cm}
                 \end{array}
                 \right.
 \end{equation}
 is called \emph{tame} canonical transformation.
\end{definition}
Note that in this general context we have no assumption of homogeneity for large $(x,\eta)$, nevertheless the mapping $\chi$ is well-defined by the global inverse
function theorem, moreover $\chi$ is a smooth bi-Lipschitz canonical transformation (i.e.\ it preserves the symplectic form)  and satisfies, for $(x,\xi)=\chi(y,\o)$,
\begin{equation}\label{B2}
|\partial_{z}^\a x_i(z)|+|\partial_{z}^\a \xi_i(z)|\leq C_\a,\quad |\a|\geq 1,\,\,z=(y,\eta),\,\,i=1,\dots,d.\end{equation}
Finally, the mapping $\chi$ enjoys  \begin{equation}\label{detcond2?}
    |\det\,\frac{\partial x}{\partial y}(y,\eta)|\geq \delta
 \end{equation}
 (that is \eqref{detcond2} for the canonical transformations of the Hamilton-Jacobi theory), which allows to uniquely determine (up to a constant) the related tame phase function $\Phi_\chi$ (see \cite[Section 2]{Wiener}).

We shall refine and apply results  for tame canonical transformations in \cite{Wiener} to the special case of the canonical transformations coming from \eqref{HS}. First, we need to introduce the class of global FIOs which are the main ingredient of this study.

\subsection{The classes $FIO(\chi,s)$ of Fourier Integral Operators}
The definition of the class $FIO(\chi,s)$ was introduced in \cite{Wiener} and can be rephrased as follows.

\begin{definition}
Let $g\in\cS(\rd)$ be a non-zero window function and $s\in\bR$. Consider   a canonical transformation $\chi$ which is a smooth bi-Lipschitz diffeomorphism and satisfies \eqref{B2}. We say that  a
continuous linear operator $T:\cS(\rd)\to\cS'(\rd)$ is in the
class $FIO(\chi,s)$ if its (continuous) Gabor matrix  satisfies the decay
condition
\begin{equation}\label{asterisco}
|\langle T \pi(w) g,\pi(z)g\rangle|\leq {C}\langle z-\chi(w)\rangle^{-s},\qquad \forall z,w\in\rdd.
\end{equation}
\end{definition}
Note that we do not require \eqref{detcond2?} to be valid.\par
The class $\fiola = \bigcup _{\chi } FIO (\chi ,s)$ is  the union of
these classes where  $\chi$ runs over the  set of all  smooth bi-Lipschitz canonical transformations satisfying \eqref{B2}.

 Gabor frames decompositions of FIOs in \cite{Wiener} produce the following issues.
\begin{itemize}
    \item [(i)] \emph{Boundedness of $T$ on $M^p(\rd)$} (\cite[Theorem 3.4]{Wiener}):\\
    If $s>2d$ and $T\in FIO(\chi,s)$, then $T$ can be extended to a bounded operator on $M^p(\rd)$ (in particular on $\lrd$).
    \item[(ii)] \emph{The algebra property} (\cite[Theorem 3.6]{Wiener}): For $ i=1,2$,  $s>2d,$
\begin{equation}\label{algebra}T^{(i)}\in FIO(\chi_i,s)\quad \Rightarrow \quad T^{(1)}T^{(2)}\in
FIO(\chi_1\circ \chi_2, s) \, .\end{equation}
 \item[(iii)] \emph{The Wiener property} (\cite[Theorem 3.7]{Wiener}): If $s>2d$, $T\in FIO(\chi,s)$  and  $T$ is invertible on $L^2(\rd)$, then $T^{-1} \in FIO(\chi^{-1},s)$.
\end{itemize}
These three properties imply that the union $\fiola$ is a Wiener subalgebra of $\cL
(\lrd )$, the class of linear bounded operators on $\lrd$. Property (ii) can be refined as follows. 
\begin{lemma} For   $s>2d,$  $T^{(i)}\in FIO(\chi_i,s)$, $ i=1,2$, the continuous Gabor matrix of the composition $T^{(1)}T^{(2)}$ is controlled  by
\begin{equation}\label{controllo1}
|\langle T^{(1)}  T^{(2)}\pi(w) g,\pi(z)g\rangle|\leq C_0 C_1 C_2 \la z-\chi_1\circ\chi_2(w)\ra^{-s}, w,z\in\rdd,
\end{equation}
where $C_i>0$ is the constant  of $T^{(i)}$ in \eqref{asterisco}, $ i=1,2$, whereas $C_0>0$ depends only on $s$ and on the Lipschitz constants of $\chi_1$ and $\chi_1^{-1}$.
\end{lemma}
\begin{proof}
Consider $g\in\cS(\rd)$ with $\|g\|_2=1$. We write the product $T^{(1)}T^{(2)}$ as
$$T^{(1)}T^{(2)} = V^*_g V_g T^{(1)}T^{(2)}
V^*_g V_g = V^*_g (V_g T^{(1)}V^*_g ) ( V_gT^{(2)} V^*_g ) V_g \, .
$$
Thus the composition of operators corresponds to the multiplication of
their (continuous) Gabor matrices.   Using the decay estimates for  the continuous Gabor matrices of $T^{(i)}$, $i=1,2$,
\begin{align}\label{asteriscoalg}
|\langle T^{(1)}  T^{(2)}\pi(w) g,\pi(z)g\rangle|&\leq \intrdd
|\langle T^{(1)} \pi(w) g,\pi(y)g\rangle| |\langle T^{(2)} \pi(y) g,\pi(z)g\rangle| dy\notag\\
&\leq C_1 C_2\intrdd \la z-\chi_1(y)\ra^{-s} \la y-\chi_2(w)\ra^{-s} dy\notag\\
&\leq C_1 C_2 C(\chi_1) \intrdd \la \chi_1^{-1}(z)-y\ra^{-s} \la y-\chi_2(w)\ra^{-s} dy\notag\\
&\leq C_1 C_2 C(\chi_1) C_s \la \chi_1^{-1}(z)-\chi_2(w)\ra^{-s}\notag\\
&\leq C_1 C_2 C(\chi_1) C_s  C(\chi^{-1}_1)\la z-\chi_1\circ\chi_2(w)\ra^{-s}
\end{align}
for every $z,w\in\rdd$, $s>2d$,  where $C_1$ and $C_2$ are the controlling constants in \eqref{asterisco}  of the operators  $T^{(1)}$ and  $T^{(2)}$, and the bi-Lipschitz property of $\chi_1$ gives
 $$\la z-\chi_1(y)\ra^{-s}\leq   C(\chi_1)   \la \chi_1^{-1}(z)-y\ra^{-s},\quad \forall y,z\in\rdd$$
 and
 $$\la \chi_1^{-1}(y)-z\ra^{-s}\leq   C(\chi^{-1}_1)   \la y-\chi_1(z)\ra^{-s},\quad \forall y,z\in\rdd.$$
 Furthemore, we used  that $v_{-s}$ is subconvolutive for $s>2d$: $v_{-s}\ast v_{-s}\leq C_s v_{-s}$ \cite[Lemma 11.1.1(d)]{book}. If we call $C_0=C(\chi_1) C_s  C(\chi^{-1}_1)$, the claim is proved.
\end{proof}

By induction we immediately obtain
\begin{corollary} For $n\in\bN$, $n\geq2$, $s>2d$, $T^{(i)}\in FIO(\chi_i,s)$,  $i=1,\dots,n$, we have
\begin{equation}\label{GMn}
|\langle T^{(1)} T^{(2)}\cdots  T^{(n)}\pi(w) g,\pi(z)g\rangle|\leq C_0C_1\cdots C_n \la z-\chi_1\circ\chi_2\circ\cdots \circ\chi_{n}(w)\ra^{-s}.
\end{equation}
where $C_0$ depends on $s$ and on the Lipschitz constants of the mappings: $$\chi_1,\chi^{-1}_1,\chi_1\circ\chi_2,(\chi_1\circ\chi_2)^{-1},\dots,\chi_1\circ\chi_2\circ\cdots\circ\chi_{n-1},(\chi_1\circ\chi_2\circ\cdots\circ\chi_{n-1})^{-1}.$$
\end{corollary}
\medskip

Observe that, using Schur's test and the same techniques as in the proof \cite[Theorem 3.4]{Wiener}, it is straightforward to obtain the following
weighted version of \cite[Theorem 3.4]{Wiener}. Hence we omit the proof.
\begin{theorem}\label{contmp} Let
$0\leq r<s-2d$, and
$\mu\in\cM_{v_r}$. For every
$1\leq p\leq\infty$, $T\in FIO(\chi,s)$ extends
to a continuous operator from
${M}^p_{\mu\circ\chi}$ into
${M}^p_{\mu}$.
\end{theorem}
Let us underline  that $\mu\circ\chi\in\cM_{v_r}$, since
 $v_r\circ\chi\asymp v_r$,  due to the bi-Lipschitz property of $\chi$.\par
If $\chi=\mathrm{Id}$, the identity operator, then the corresponding Fourier integral operators are simply pseudodifferential operators, as already shown in \cite{GR}.  The characterization below is written for \psdo s in the Kohn-Nirenberg form $\sigma(x,D)$, but it works the same for any  $\tau$-form (in particular Weyl form $\sigma^w(x,D)$)  in which is written a pseudodifferential operator.
 \begin{proposition}
   \label{charpsdo}
 Fix $g \in\cS(\rd)$ and let $\sigma \in \cS'(\rdd )$. For $s\in\bR$, the symbol $\sigma \in
 M^{\infty}_{1\otimes v_s}(\rdd )$ \fif\
 \begin{equation}
   \label{eq:kh9}
   |\langle \sigma(x,D) \pi (w) g , \pi (z) g\rangle | \leq C \la z-w\ra^{-s} \qquad
   \forall w,z \in \rdd \, .
 \end{equation}
 \end{proposition}
 \par
 Similarly, under additional assumptions on  the classes $FIO(\chi,s)$, their operators can be written in the following integral form, called \emph{FIOs of type I}:
 \begin{equation}\label{sei}
 I(\sigma,\Phi) f(x)=\int_{\rd} e^{2\pi i
  \Phi(x,\eta)}\sigma(x,\eta)\widehat{f}(\eta)\,d\eta,  \quad f\in\cS(\rd),
 \end{equation}
 where  $\sigma\in M^{\infty}_{1\otimes v_s}(\rdd)$) and $\Phi$ a tame phase function. This particular form is allowed starting from the class $FIO(\chi,s)$ whenever  the mapping $\chi$ enjoys the additional property  \eqref{detcond2?} as explained in the following characterization \cite[Theorem 4.3]{Wiener}.

\begin{theorem}\label{caraI}
Consider $g\in\cS(\rd)$ and   $s\geq 0$.
Let $I$ be a continuous linear operator $\cS(\rd)\to\cS'(\rd)$ and
$\chi$ be a tame canonical transformation satisfying \eqref{detcond2?}.  Then the following properties are
equivalent. \par {\rm (i)} $I=I(\sigma,\Phi_\chi)$ is a FIO of type
I for some $\sigma\in M^{\infty}_{1\otimes v_s}(\rdd)$. \par {\rm
(ii)} $I\in FIO(\chi ,s)$.\par
\end{theorem}
Moreover, gluing together the results \cite[Theorem 3.3]{locNC09} and \cite[Theorem 4.3]{Wiener} we observe that
the constant $C$ in \eqref{asterisco} satisfies
\begin{equation}\label{cost}
C\asymp \|\sigma\|_{M^{\infty}_{1\otimes v_s}}.
\end{equation}
For $\chi={\rm Id}$ we recapture the characterization for \psdo s of Proposition \ref{charpsdo}.

Since we shall apply our results to the mappings $\chi_t(x,\eta)$ coming from the Hamilton-Jacobi system \eqref{HS}, we need to be more precise on the estimate \eqref{cost}: it is important to see how  the constants involved in the equivalence depend on the time variable $t$.
It amounts rewriting the proofs of the results cited above for the special case of a phase function
$\Phi\in \cC^\infty(]-T,T[, \Gamma^2(\rdd))$ and following the time variable $t$.  We state the result here and we refer to the Appendix for a sketch  the main points of the proofs, leaving the details to the interested reader.
\begin{theorem}\label{caraIt}
Consider $g\in\cS(\rd)$,   $s\geq 0$, and $T>0$ such that in $]-T,T[$ the equation \eqref{eiconal} is solved by the tame phase $\Phi\in \cC^\infty(]-T,T[, \Gamma^2(\rdd))$. Let   $\chi_t$ be the related tame canonical transformation in \eqref{rel-chi-phi}. Let $I$ be a continuous linear operator $\cS(\rd)\to\cS'(\rd)$. Then the following are equivalent:
\par {\rm (i)} $I=I(\sigma_t,\Phi_{\chi_t})$ is a FIO of type
I for some $\sigma_t\in M^{\infty}_{1\otimes v_s}(\rdd)$ such that
\begin{equation}
\label{controllosigma}
\|\sigma_t\|_{ M^{\infty}_{1\otimes v_s}}\leq H(t)\in\cC(]-T,T[).
\end{equation} \par {\rm
(ii)} $I\in FIO(\chi_t,s)$
 and the constant $C=C(t)$ in \eqref{asterisco} is in $\cC(]-T,T[)$.
\end{theorem}

\section{Unperturbed Schr\"odinger Equations}

The previous theory applies in the study of the Cauchy problem for linear Schr\"odinger equations. First, consider the unperturbed case:
\begin{equation}\label{cpA}
\begin{cases}
i\partial_t u+A u=0\\
u(0,x)=u_0(x),
\end{cases}
\end{equation}
with $x\in\rd$, $u_0\in\cS(\rd)$. The operator $A=a(x,D)\in G^{2,cl}$ is a formally  self-adjoint pseudodifferential operator in the Kohn-Nirenberg form. This means that the symbol $a\in \Gamma^{2,cl}(\rdd)$ has the expansion
\begin{equation}\label{poly}
a(x,\xi)\sim \sum_{j=0}^\infty {a}_{2-j}(x,\xi),
\end{equation}
where the principal symbol $a_2(x,\xi)$ is real-valued,
since $A$ is self-adjoint. The problem \eqref{cpA} is forward and
backward well-posed in $\cS(\rd)$ and the corresponding evolution operator
$e^{it A}$, acting from $\cS(\rd)$ into $\cS(\rd)$, extends to
$L^2$-isometries \cite{helffer84}.

The classical evolution \eqref{HS}
has the solution $(x(t),\xi(t))=\chi_t(y,\eta)$ in \eqref{mappachi} and for a suitable $T>0$ and  $t\in]-T,T[$
the evolution operator $e^{it A}$ can be well approximated by a FIO of type I, as expressed in  \cite[Proposition 3.1.1]{helffer84} for the special case of elliptic operators (that is operators whose corresponding principal symbols satisfy \eqref{ellitticsimb}), but still valid  without the assumption \eqref{ellitticsimb}, as observed in \cite[Section 5.3]{bertinoro}). In our framework the result \cite[Proposition 3.1.1]{helffer84} can be rephrased as follows.
\begin{proposition}\label{Prop31}
Given the  Cauchy problem \eqref{cpA} with $a(x,D)$ as above, then there exists a $T>0$, a symbol $\sigma(t,x,\eta)\in \cC^\infty(]-T,T[,\Gamma^{0}(\rdd))$ a real-valued phase function $\Phi\in\cC^\infty(]-T,T[, \Gamma^2(\rdd))$ satisfying
\eqref{eiconal} and \eqref{detmisto} such that the evolution operator can be written as
 \begin{equation}\label{soluzioneA}
 (e^{i t A}u_0)(t,x)=(F_tu_0)(t,x)+(R_t u_0)(t,x),
 \end{equation}
 where $F_t$ is the FIO of type I
 \begin{equation}\label{FIO1}
 (F_tu_0)(t,x)=\intrd e^{2\pi i\Phi(t,x,\eta)} \sigma(t,x,\eta){\widehat
  {u_0}}(\eta)d\eta
  \end{equation}
 and the operator $R_t$ has kernel in $\cC^\infty(]-T,T[,\cS(\rdd))$ (thus $R_t$ is regularizing, i.e., $R_t: \cS'(\rd) \to
 \cS(\rd)$.)
\end{proposition}
This result says that in an interval $]-T,T[$ the propagator $e^{i t A}$ can be represented by a type I FIO $F_t$  up to an error, which however is a regularizing operator.
\begin{remark}
We observe that the function $\Phi(t,\cdot)$ of Proposition \eqref{Prop31} and the related canonical transformation $\chi_t$ in \eqref{mappachi} are tame, with Lipschitz constants of $\chi_t$ and its inverse that can be controlled by a continuous function of $t$ on the interval $]-T,T[$ and so can be chosen uniform with respect to $t$ on $]-T,T[$.\par
\end{remark}

We will show that if we replace the type I FIO $F_t$ by a more general operator in the classes $FIO(\chi_t,s)$, we are able to remove the error $R_t$ in \eqref{soluzioneA}. Precisely, we can state the following issue.
\begin{proposition}\label{cor31}
Under the assumptions of Proposition \ref{Prop31} we have \begin{equation}
\label{T1}
e^{i t A}\in\cap_{s\geq0}FIO(\chi_t,s),\quad t\in]-T,T[,\end{equation}
where  $\chi_t$ is defined in \eqref{mappachi}. Moreover for every $s\geq0$ there exists $C(t)=C_s(t)\in \cC(]-T,T[)$ such that, for every $g\in\cS(\rd)$ the Gabor matrix satisfies
\begin{equation}\label{GM1}
|\la e^{i t A} \pi(w)g, \pi(z)g\ra|\leq C(t)  \la z-\chi_t(w)\ra^{-s},\quad w,z\in\rdd.
\end{equation}
\end{proposition}
\begin{proof}  By  Proposition \ref{Prop31} there exists a $T>0$ such that  the evolution $e^{itA}$ can be written as \eqref{soluzioneA}, where $F_t$ is a type I FIO with symbol in
 $\sigma(t,x,\eta)\in \cC^\infty(]-T,T[,\Gamma^{0})$ and phase  $\Phi(t,\cdot)$ in \eqref{eiconal}.
Since
\begin{equation*}
\cC^\infty(]-T,T[,\Gamma^{0})\subset \cC^\infty(]-T,T[, S^0_{0,0})=\cC^\infty(]-T,T[, \cap_{s\geq 0} M^\infty_{1\otimes v_s}),
\end{equation*}
we can find $C>0$ and $k_1=k_1(s)\in\bN$ such that
\begin{equation*}
\| \sigma(t,\cdot)\|_{M^\infty_{1\otimes v_s}}\leq C  \sum_{|\a|\leq k_1}\|\partial^\a\sigma(t,\cdot)v_{|\a|}\|_\infty
\end{equation*}
where $\sum_{|\a|\leq k_1}\|\partial^\a\sigma(t,\cdot)v_{-|a|}\|_\infty\in\cC(]-T,T[)$ by assumption. Hence
the characterization of Theorem \ref{caraIt} gives  $F_t\in \cap_{s\geq 0} FIO(\chi_t,s)$, $t\in ]-T,T[$, where the canonical transformation $\chi_t$ is defined in \eqref{mappachi} and related with $\Phi(t,\cdot)$ by \eqref{rel-chi-phi}
and
\begin{equation*}
|\la F_t \pi(w)g, \pi(z)g\ra|\leq C(t) \la z-\chi_t(w)\ra^{-s}
\end{equation*}
with $C(t)\in\cC (]-T,T[)$.

Fix now $g\in\cS(\rd)$ with $\|g\|_2=1$ so that the inversion formula \eqref{treduetre} becomes ${\rm Id}=V_g^\ast V_g$ and we can write
$R_t= V_g^\ast V_g R_t V_g^\ast V_g$. Since $R_t$ is a regularizing operator, for
$T_t:=V_g R_t V_g^\ast$,
the following diagram is
commutative:
\[
\begin{diagram}
\node{\cS'(\rdd)}\arrow{s,l}{V_g^\ast} \arrow{c,t}{T_t}  \node{ \cS(\rdd)}
  \\
\node{\cS'(\rd)}  \arrow{c,t}{R_t} \node{ \cS(\rd)} \arrow{n,r}{V_g}
\end{diagram}
\]
(see the definition an properties of $V_g$ and its adjoint $V_g^\ast$ in Subsection \ref{2.2}).
This means that the linear operator $T_t: \cS'(\rdd)\to\cS(\rdd) $ is regularizing as well and so its kernel $k_t(w,z)=\la R_t \pi(w)g, \pi(z)g\ra \in \cS(\bR^{4d})$ satisfies
\begin{equation}\label{kernel}|k_t(w,z)|=|\la R_t \pi(w)g, \pi(z)g\ra|\leq \la z\ra^{-N} \la w\ra^{-N},\quad \forall z,w\in\rdd, \,\,\forall N\in \bN.
\end{equation}
The previous estimates yields $R_t\in FIO (\chi,s)$, for every  bi-Lipschitz mapping $\chi$ and every $s\geq0$. Indeed, $$\la z-\chi(w)\ra\leq \la z\ra \la \chi(w)\ra\asymp \la z\ra \la w\ra$$
and choosing $N\geq s$ in \eqref{kernel} we obtain $$|k_t(w,z)|\lesssim \la z\ra^{-N} \la w\ra^{-N}\lesssim \la z-\chi(w)\ra^{-s}.$$
Finally, if $\sigma(R_t)(z)$ is the Kohn-Nirenberg symbol  of $R_t$, using the fact that $\cS(\rdd)\subset S_{0,0}^0$ with continuous embedding for every $s\geq0$ we   find $C>0$ and $k_2\in\bN$ such that
 $$\|\sigma(R_t)\|_{M^\infty_{1\otimes v_s}}\leq C \sum_{|\a+\b|\leq k_2}\|z^\a \partial_z^\b \sigma(R_t)(z)\|_\infty\in\cC (]-T,T[).$$
Using Theorem \ref{caraIt} with $\chi_t$ in \eqref{mappachi} which is tame for $t\in ]-T,T[$,   we find $C(t)\in\cC(]-T,T[)$ such that
 $$|\la R_t\pi(w)g,\pi(z)g \ra| \leq C(t)\la z-\chi_t(w)\ra^{-s}.$$
Finally the thesis follows since  $FIO(\chi,s)$ are linear spaces:
\begin{align*}
|\la e^{i t A}\pi(w)g, \pi(z)g\ra|&\leq |\la F_t \pi(w)g, \pi(z)g\ra|+|\la R_t\pi(w)g, \pi(z)g\ra|\\
&\leq C(t) \la z-\chi_t(w)\ra^{-s}.
\end{align*}
which gives \eqref{GM1}.
\end{proof}

The previous proposition gives an approximation of $e^{it A}$ for $|t|<T$. Using the group property of the propagator $e^{it A}$ Helffer in \cite[page 139]{helffer84} describes how to obtain an approximation of $e^{it A}$ for every $t\in\bR$.
Indeed, a classical trick, jointly with the group property of   $e^{it A}$, applies. We consider $T_0<T/2$ and define
$$I_h=]h T_0, (h+2)T_0[, \quad h\in \bZ.
$$
For $t\in I_h$, by the group property of  $e^{it A}$:
\begin{equation}\label{eita}
e^{it A}=e^{i(t-hT_0) A}(e^{i(hT_0) A/|h|})^{|h|}
\end{equation}
and using Proposition \ref{Prop31}, one can write
\begin{equation}\label{eita2}e^{it A}-F_{t-hT_0}(F_{\frac{h}{|h|}T_0})^{|h|}\in\cC^\infty(I_h,\cL (\cS',\cS)).
\end{equation}
In general, $e^{it A}$ or even the composition  $F_{t-hT_0}(F_{\frac{h}{|h|}T_0})^{|h|}$ cannot be represented as a type I FIO in the form \eqref{FIO1}.  We shall prove below that
the evolution $e^{it A}$ is in the class  $\cap_{s\geq 0}FIO(\chi_t,s)$ for every $t\in\bR$, with $\chi$  defined in \eqref{mappachi}, so that this class is proven to be the right framework for describing the evolution  $e^{it A}$.

\begin{theorem}\label{UNCP}
Given the  Cauchy problem \eqref{cpA} with $A=a(x,D)$ as above, consider the mapping   $\chi_t$ defined in \eqref{mappachi}.
 Then  \begin{equation}\label{A1}e^{it A}\in \cap_{s\geq0} FIO(\chi_t,s)),\quad t\in\bR\end{equation}
 and for every $s>2d$ there exists $C(t)\in \cC(\bR)$ such that
 \begin{equation}\label{GM2}
 |\la e^{i t A} \pi(w)g, \pi(z)g\ra|\leq C(t)  \la z-\chi_t(w)\ra^{-s},\quad w,z\in\rdd,\quad t\in\bR.
 \end{equation}
\end{theorem}
\begin{proof} We fix $T_0<T/2$ as above. For $t\in\bR$, there exists a $h\in\bZ$ such that $t\in I_h$. Using Proposition \ref{cor31} for $t_1=t-hT_0\in ]-T,T[$ we have that $e^{it_1A}\in FIO(\chi_{t_1},s)$ and for $t_2=\frac{h}{|h|}T_0 \in ]-T,T[$,
$e^{it_2A}\in FIO(\chi_{t_2},s)$, for every $s\geq 0$, and there exists a continuous function $C(t)$ on $]-T,T[$ such that \eqref{GM1} is satisfied for $t=t_1$ and $t=t_2$. Using the algebra property \eqref{algebra}, for every $s>2d$,
$$e^{it_1A}(e^{it_2A})^{|h|}\in FIO(\chi_{t_1}\circ(\chi_{t_2})^{|h|},s)$$
and  the group law \eqref{prodotto} for $\chi_t(y,\eta)=S_0(t)(y,\eta)$ gives
$$
\chi_{t_1}\circ(\chi_{t_2})^{|h|}=\chi_{t_1+|h| t_2}=\chi_{t},
$$
as expected and using \eqref{GMn} we obtain that the Gabor matrix of the product $e^{it_1A}(e^{it_2A})^{|h|}$ is controlled by a continuous function $C_h(t) $ on $I_h$. Finally, from  the estimate
$$ |\la e^{it A} \pi(w)g, \pi(z)g\ra|\leq C_h(t) \la z-\chi_t(w)\ra^{-s},\quad t\in I_h,$$
with $C_h\in\cC(I_h)$, it is easy to construct a new continuous controlling function $C(t)$ on $\bR$ such that \eqref{GM2} is satisfied.
\end{proof}
\section{Schr\"odinger Equations with bounded perturbations}
We now study the Cauchy problem
  for linear Schr\"odinger
  equations of the type
  \begin{equation}\label{C1}
\begin{cases} i \displaystyle\frac{\partial
u}{\partial t} +H u=0\\
u(0,x)=u_0(x),
\end{cases}
\end{equation}
with $t\in\bR$ and the initial condition $u_0\in\cS(\rd)$.  We consider  a Hamiltonian of the form
\begin{equation}\label{C1bis}
H=a(x,D)+ \sigma(x,D),
\end{equation}
 where $A=a(x,D)$ is the pseudodifferential operator satisfying \eqref{cpA}, whose corresponding propagator $e^{it A}\in \cap_{s\geq 0} FIO(\chi_t,s)$, for $t\in\bR$, as shown in the preceding section.

 The perturbation   $B=\sigma(x,D)$ is a pseudodifferential operator  with a
symbol $\sigma\in M^{\infty}_{1\otimes v_s}(\rdd)$, $s>2d$. This last requirement implies the boundedness of $B$ on $M_{\mu}(\rd)$ for a weight $\mu$ as in the assumptions of Theorem \ref{contmp} (with $\chi={\rm Id}$), (see also \cite{charly06} using $M^{\infty}_{1\otimes v_s}(\rdd)\subset M^{\infty,1}(\rdd)$, $s>2d$) and in particular on $L^2(\rd)$.
Hence,  $H = A+ B$ is a
bounded perturbation of the generator $A $  of a unitary
group by~\cite{RS75}, and $H$ is the generator of a well-defined
(semi-)group.  We shall heavily use the theory of operator
semigroups, addressing to the
textbooks~\cite{RS75} and \cite{EN06} for an introduction on the topic. Our result, containing Theorem \ref{T1.1},  is as follows.
\begin{theorem}\label{PCP}
Let $s>2d$. Consider the  Cauchy problem \eqref{C1} with $A=a(x,D)$ and $B=\sigma(x,D)$ as above. Let  $\chi_t$ be the mapping defined in \eqref{mappachi}.
 Then the solution can be written as
 $$e^{it H}=e^{it A}Q(t)=\tilde{Q}(t)e^{it A}\in FIO(\chi_t,s), \quad t\in \bR,$$
where $Q(t)$ and $\tilde{Q}(t)$ are pseudodifferential operators  with symbols in $M^{\infty}_{1\otimes v_s}(\rdd)$ and the continuous Gabor matrix satisfies
$$|\la e^{it H} \pi(w)g,\pi(z)g\ra |\leq C(t) \la z-\chi_t(w)\ra^{-s},\quad w,z\in\rdd,
$$
for a suitable positive continuous function $C(t)$ on $\bR$.
\end{theorem}
\begin{proof}
The pattern is similar to \cite[Theorem 4.1]{CGNR13}. We show the result on the interval  $[0,+\infty[$, for the interval $]-\infty,0]$ the result is obtained by the previous case by replacing $t$ with $-t$.\par The operator $A$ is the generator of a strongly continuous one-parameter group  on $L^2(\rd)$ and $T(t) = e^{it A} $ is the corresponding (semi)group that solves the evolution equation $i \frac{dT(t)}{dt} = A T(t)$. Then $e^{itA}$ is a strongly continuous one-parameter group on
$L^2(\rd)$. As already observed, by the assumptions on the symbol of $B$, it follows that $B$
is a bounded operator on $L^2(\rd)$, hence $H=A +B$ is the
generator of a strongly continuous one-parameter group $S(t)$~\cite{EN06}.
 The perturbed semigroup $S(t) = e^{itH}$ satisfies an  abstract Volterra
equation
\begin{equation}
  \label{eq:kh6}
  S(t)f= T(t)f + \int _0^t T(t-s) B S(s) f \, ds =
  T(t)\Big(\mathrm{Id} + \int _0^t T(-s) B T(s) T(-s) S(s)  \, ds
  \Big)f
\end{equation}
for every $f\in L^2(\rd) $ and $t\geq 0$. If we define
by $Q(t) = T(-t)S(t)$, then by \eqref{eq:kh6} $Q(t)$  satisfies the Volterra equation
\begin{equation}
  \label{eq:kh7}
  Q(t) =  \mathrm{Id} + \int _0^t T(-s) B T(s) Q(s)   \, ds  \, ,
\end{equation}
where the integral is to be understood in the strong sense.
Now write $B(s) = T(-s) B T(s)$, then the solution of~\eqref{eq:kh7}
can be written as a so-called \emph{Dyson-Phillips expansion}
(\cite[X.69]{RS75} or \cite[Ch.~3, Thm.~1.10]{EN06})
\begin{equation}
  \label{eq:kh8}
  Q(t) = \mathrm{Id} + \sum _{n=1}^\infty (-i)^n \int _0^t \int
  _0^{t_1} \dots \int
  _0^{t_{n-1}} B(t_1) B(t_2) \dots B(t_n) \, dt_1 \dots dt_n := \sum
  _{n=0}^\infty Q_n(t) \, .
\end{equation}

We shall show that $Q(t)$ is a pseudodifferential operator
with symbol in $M^{\infty}_{1\otimes v_s}(\rdd)$.\par

For $\tau\in [0,t]$,  the algebra property \eqref{algebra} gives $$B(\tau) = e^{i(-\tau) A} B e^{i\tau A} \in FIO(\chi_{-\tau}\circ {\rm Id} \circ \chi_\tau,s )=FIO({\rm Id},s)$$
 since $\chi_{-\tau}\circ {\rm Id} \circ \chi_\tau=\chi_{-\tau} \circ \chi_\tau=S_0(0)={\rm Id}$ by \eqref{prodotto}.
 Moreover, $e^{i\pm \tau A}$ satisfies \eqref{GM2}, so that  using
\eqref{GMn} with $n=3$, $T^{(1)}= e^{i(-\tau) A}$, $T^{(2)}=B$, $T^{(3)}= e^{i\tau A}$ and $\chi={\rm Id}$ we can write
\begin{equation}\label{funzcost}|\la B(\tau)\pi(w)g,\pi(z)g\ra|\leq C(\tau)\la z-w\ra^{-s},
\end{equation}
for a new continuous function $C(\tau)$ on  $\bR$.
Using  \eqref{GMn} again for the composition of \psdo s
$\prod_{j=1}^{n} B(t_j)$ we obtain
\begin{align*}
|\langle \prod _{j=1}^n B(t_j) \pi (w) g, \pi (z) g\rangle | 
&\leq C_0 C(t_1)   \cdots  C(t_n) \la z-w\ra^{-s},
\end{align*}
with $C(t)\in\cC(\bR)$ in \eqref{funzcost}. \par
We now show that $Q_n(t) $ is a \psdo \, with symbol in $M^\infty_{1\otimes v_s}(\rdd)$.
We  control the Gabor matrix of $Q_n(t) $ as follows:
$$
|\langle Q_n(t) \pi (w)g, \pi (z)g\rangle|\leq C_0 \int _0^t \int
  _0^{t_1} \dots \int
  _0^{t_{n-1}}    C(t_1)   \cdots C(t_n)  dt_1 \ldots dt_n \la z-w\ra^{-s}.
$$

If we define $$H(t)=\max_{\tau\in [0,t]} C(\tau)\in\cC(\bR),$$
we obtain
$$
|\langle Q_n(t) \pi (w)g, \pi (z)g\rangle|\leq C_0 H(t)^n
  \frac{t^n }{n!} \la z-w\ra^{-s} .$$
Finally, setting $\tilde{H}(t)=t H(t)\in\cC(\bR)$,
\begin{align*}
|\langle Q(t) \pi (w) g , \pi (z)g\rangle | & \leq   \sum
_{n=0}^\infty |\langle Q_n(t) \pi (w) g , \pi (z)g\rangle |  \la z-w\ra^{-s}\\
 & \leq   C_0\sum _{n=0}^\infty
   \frac{\tilde{H}(t)^n }{n!} \la z-w\ra^{-s}\\
   &=C(t)  \la z-w\ra^{-s},
   \end{align*}
for a new function $C(t)\in\cC(\bR)$.
This gives by Theorem \ref{caraIt} that $Q(t)\in FIO({\rm Id},s)$ that is $Q$ is a \psdo \,with symbol in $M^\infty_{1\otimes v_s}(\rdd)$. Finally, the algebra property again gives
$$
e^{itH} = T(t) C(t) \in FIO(\chi_t,s),
$$
and the estimate \eqref{controllo1} gives that the Gabor matrix of $e^{itH}$ is controlled by a continuous function $C(t)$ on $\bR$.
\end{proof}

Consequently, the Schr\"odinger equation preserves the phase-space
concentration, as expressed by the following issue.

\begin{corollary}\label{corfinal}
 Let
 $0\leq r<s-2d$, and
 $\mu\in\cM_{v_r}$. If the initial condition $u_0\in M^p_{\mu\circ\chi}$, $1\leq p\leq \infty$, then
 $u(t,\cdot )
 = e^{itH}u_0\in  M^p_{\mu}$, for all $t\in\bR$.
 \end{corollary}
 \begin{proof}
 It follows immediately  from Theorems \ref{PCP} and \eqref{contmp}.
 \end{proof}

 Using $v_r\circ \chi\asymp v_r$, we observe that the Schr\"odinger evolution
 preserves the phase space concentration $M^p_{v_r}$ of the initial condition $u_0$. In other words, the time evolution leaves $M^p_{v_r}$ invariant.
 \begin{corollary}\label{corfinal2}
  Let  $ |r|<s-2d$. If the initial condition $u_0\in M^p_{v_r}$, $1\leq p\leq\infty$, then
  $u(t,\cdot )
  = e^{itH}u_0\in  M^p_{v_r}$, for all $t\in\bR$.
  \end{corollary}
 \begin{proof}
  The result is a special case of Corollary \ref{corfinal}
  once we  prove, for $r>0$,  $v_{q}$ is $v_r$-moderate if and only if $|q|\leq r$. But this is an easy consequence of Peetre's inequality
  $$ \la z+\zeta\ra^q\leq \la z\ra^{|q|}\la \zeta\ra^q.
  $$
  \end{proof}
  \par
  From Corollaries \ref{corfinal} and \ref{corfinal2} we recapture \eqref{T1.5eq1} in Theorem \ref{T1.5}.
\section{Propagation of singularities}
In what follows we shall use $\chi_t$ for $\chi_t$ when it is more convenient.
\begin{proposition}\label{5.2}
Let $f\in\cS'(\rd)$,  $r>0$. Then:\\
(i) The definitions of  $\wpr (f)$  and $WF_G (f)$  do not depend on the choice of the window $g$.\\
(ii) $f\in M^p_{v_r}(\rd)$ if and only if $\wpr (f)=\emptyset$. Similarly, $f\in\cS(\rd)$ if and only if $WF_G (f)=\emptyset$.
\end{proposition}

\noindent
The proof of $(i)$ will be given later, as a consequence of more general arguments. The proof of $(ii)$ follows easily from the compactness of the sphere $\cS^{2d-1}$ and \eqref{HCbis}.

The following statement gives the second part of Theorem \ref{T1.5}.
\begin{theorem}\label{5.3}
Under the assumptions of Theorem \ref{PCP}, for $u_0\in M^p_{v_{-r}}(\rd)$, $1\leq p\leq\infty$, $0<2 r<s-2d$, we have
\begin{equation}\label{5.2eq}
\wpr (e^{i t H}u_0)=\chi_t(\wpr( u_0)).
\end{equation}
\end{theorem}
\begin{proof} We shall prove that $\wpr (e^{i t H}u_0)\subset\chi_t(\wpr (u_0))$ for any $t\in\bR$. Then, by applying the inclusion to $v_0= e^{-i t H}u_0$, the opposite inclusion will follow, and \eqref{5.2eq} will be proved.
\par
Fixed $t\in\bR$, we assume $z_0\notin \chi_t(\wpr( u_0))$. Since $\chi_t$ is a homogeneous diffeomorphism for large $|z|$, this is equivalent to say that $w_0=\chi_t^{-1}(z_0)$ does not belong to  $\wpr (u_0)$. Therefore for a sufficiently small open conic neighborhood $\Gamma_{w_0}\subset\rdd\setminus{0}$ of $w_0$ we have
\begin{equation}\label{5.3eq}
\int_{\Gamma_{w_0}}|V_g u_0(w)|^p\la w\ra^{pr}<\infty.
\end{equation}
Note also that, in view of the assumption $u_0\in M^p_{v_{-r}}(\rd)$, we have
\begin{equation}\label{5.4eq}
\intrd |V_g u_0(w)|^p \la w\ra^{-pr}<\infty.
\end{equation}
Now from Theorem \ref{PCP}, we have
\begin{equation}\label{5.5eq}
V_g (e^{itH} u_0)(z)=\intrdd k(t,w,z) V_g u_0(w) dw
\end{equation}
with
\begin{equation}\label{5.6eq}
|k(t,w,z)|\leq C(t)  \la z-\chi_t(w) \ra^{-s},\quad w,z\in\rdd.
\end{equation}
We have to show that $z_0\notin \wpr(e^{i t H} u_0)$. To this end, take an open conic neighborhood $\Gamma'_{z_0}$ of $z_0$, such that
$\overline{\Gamma_{z_0}^{'}} \subset \chi_t(\Gamma_{w_0})$. This implies that for $z\in \Gamma_{z_0}^{'}$ and $w\notin\Gamma_{w_0}$
we have
\begin{equation}\label{5.7eq}
 \la z-\chi_t(w) \ra\gtrsim  \max\{\la z\ra,  \la w \ra\},
\end{equation}
since $\chi_t$ is a Lipschitz diffeomorphism. Using \eqref{5.5eq} and \eqref{5.6eq} we estimate
\begin{equation}\label{5.8eq}
|\la z\ra^r V_g (e^{itH} u_0)(z)|\lesssim \intrdd I(z,w) \,dw,
\end{equation}
with
\begin{equation}\label{eq5.9}
I(z,w)= \la z \ra^r\la z-\chi_t(w)\ra^{-s}|V_g u_0(w)|.
\end{equation}
To show $z_0\notin WF_G^r(e^{itH} u_0)$ it will be sufficient to show that $$ \left\| \intrdd I(\cdot,w)\,dw\right\|_{L^p(\Gamma'_{z_0})}<\infty.$$
First, we estimate $\intrdd I(z,w)\,dw$ for $z\in \Gamma'_{z_0}$. We split the domain of integration into two domains $\Gamma_{w_0}$ and $\rdd\setminus \Gamma_{w_0}$. In $\rdd\setminus \Gamma_{w_0}$ we   use \eqref{5.7eq} to obtain
\begin{align*}\int_{\rdd\setminus \Gamma_{w_0}} I(z,w)\,dw &\leq \int_{\rdd\setminus \Gamma_{w_0}} \la z \ra^r\la w \ra ^r\la z-\chi_t(w)\ra^{-s}\frac{|V_g u_0(w)|}{\la w \ra ^r}\, dw\\
&\lesssim  \int_{\rdd\setminus \Gamma_{w_0}} \la z-\chi_t(w)\ra^{2r-s}\frac{|V_g u_0(w)|}{\la w \ra ^r}\, dw\\
&\lesssim \left(\la \cdot\ra ^{2r -s}\ast \frac{|V_g u_0(\cdot)|}{\la \cdot \ra ^r}\right)(z).
\end{align*}
So by \eqref{5.4eq} and using $2r-s<-2d$,
$$\left\|\int_{ \rdd\setminus \Gamma_{w_0}} I(\cdot,w)\,dw \right \|_{L^p(\Gamma'_{z_0})}\lesssim \|\la \cdot\ra ^{2r -s}\|_{L^1(\rdd)}\|\,|V_g u_0|\la \cdot \ra ^{-r}\|_{L^p(\rdd)}<\infty.$$
 In the domain $\Gamma_{w_0}$, 
we have
\begin{align*}\int_{ \Gamma_{w_0}} I(z,w)\,dw &\leq \int_{\Gamma_{w_0}} \la z \ra^r\la w \ra ^{-r}\la z-\chi_t(w)\ra^{-r}\la z-\chi_t(w)\ra^{r-s}|V_g u_0(w)|\la w \ra ^r\, dw\\
&\lesssim  \int_{ \Gamma_{w_0}} \la z-\chi_t(w)\ra^{r-s}|V_g u_0(w)|\la w \ra ^r\, dw\\
&\lesssim \la \chi_t^{-1}(\cdot)\ra ^{r -s}\ast \left(\mbox{Char}_{\Gamma_{w_0}}\cdot|V_g u_0
|\la \cdot \ra ^r\right)(z)
\end{align*}
where $\mbox{Char}_{\Gamma_{w_0}}$ is the characteristic function of the set $\Gamma_{w_0}$. The assumption  \eqref{5.3eq} yields to the estimate
\begin{align*}\left\|\int_{ \Gamma_{w_0}} I(\cdot,w)\,dw \right \|_{L^p(\Gamma'_{z_0})}&\lesssim \|\la \chi_t^{-1}(\cdot)\ra ^{r -s}\|_{L^1(\rdd)}\| \,|V_g u_0|\la \cdot \ra ^r\|_{L^p(\Gamma_{w_0})}\\
&\asymp \|\la \cdot\ra ^{r -s}\|_{L^1(\rdd)}\| V_g u_0\la \cdot \ra ^r\|_{L^p(\Gamma_{w_0})} <\infty,
\end{align*}
for $\chi_t$ is a bi-Lipschitz diffeomorphism and $r-s<2 r-s<-2d$.
This concludes the proof.
\end{proof}
\par
The preceding arguments apply with small changes in the  proof of \eqref{WT1.3}. Let us detail the proof for sake of clarity.
\begin{proof}[Proof of \eqref{WT1.3}]
As in the previous proof, it is enough to show $WF_G (e^{i t H}u_0)\subset \chi_t(WF_G (u_0))$ for any $t\in\bR$. We have to prove that for every $u_0\in\cS'(\rd)$ and $z_0\in\rd$, $z_0\not=0$, the assumption $z_0\notin \chi_t(WF_G u_0)$ implies $z_0\notin WF_G(e^{i t H}u_0)$. Arguing as before, we have that the estimates \eqref{5.3eq} are satisfied for every $r>0$ in a cone $\Gamma_{w_0}$ independent of $r$. Now recall from \eqref{HCbis} that $\cS'(\rd)=\bigcup_{s\geq 0}M^\infty_{v_s}(\rd)$. Therefore $u_0\in M^\infty_{v_{-r_0}}(\rd)$  for some $r_0\geq 0$. Since $\sigma\in S^0_{0,0}=\bigcap_{s\geq 0}M^{\infty}_{1\otimes v_{-s}}(\rdd)$ by \eqref{HC}, we have $\sigma\in M^{\infty}_{1\otimes v_s}(\rdd)$ for every $s\geq0$. We may then apply the arguments in the preceding proof with $s>2r+2d>2r_0+2d$ and obtain the expected estimates \eqref{WFSeq} for any $r>0$. By observing that the choice of the cone $\Gamma'_{z_0}$ does not depend on $r$, the proof is concluded.
\end{proof}

 \begin{proof}[Proof of Proposition \ref{5.2}, (i)]
We prove the independence of the definition of  $\wpr (f)$ on the choice of the window $g$. The independence of  $WF_G (f)$ is attained similarly. \par We assume the estimate for $V_g f$ \eqref{5.1} satisfied, for some fixed $g\in\cS(\rd)\setminus\{0\}$ and some conic neighborhood $\Gamma_{z_0}$ and we want to prove that the estimate holds for $V_h f$, where $h\in\cS(\rd)\setminus\{0\}$ is fixed arbitrary, after possibly shrinking $\Gamma_{z_0}$. To this end, we use Lemma \ref{changewind} which gives
  $$|V_{h} f(z)|\lesssim (|V_{g} f|\ast|V_{h} g|)(z). $$
  Since $V_{h} g\in\cS(\rdd)$ for $g,h\in\cS(\rd)$, we have that for every $s\geq 0$
   $$|V_{h} f(z)|\lesssim \intrdd \la z-w\ra^{-s} |V_{g} f|(w)\,dw. $$
   We know that $f\in M^p_{v_{-r_0}}(\rd)$, for some $r_0\geq 0$. Taking then $s>\max\{r,r_0+2d\}$, the arguments in the proof of Theorem \ref{5.3} apply with $\chi_t=$Id, $w_0=z_0$.
 \end{proof}
 \begin{proposition}\label{5.4}
Let $\sigma \in  M^\infty_{1\otimes v_s}(\rdd)$, $s>2d$ and $0<2r<s-2d$. Then for every $f\in M^p_{v_{-r}}(\rd)$ we have
\begin{equation}\label{5.10eq}
\wpr \,(\sigma(x,D)f)\subset \wpr (f).
\end{equation}
If $\sigma\in S^0_{0,0}$, then for every $f\in\cS'(\rd)$,
\begin{equation}\label{5.11eq}
WF_G  \,(\sigma(x,D)f)\subset WF_G(f).
\end{equation}
 \end{proposition}
 \begin{proof}
 If $\sigma \in M^\infty_{1\otimes v_s}(\rdd)$, then from Proposition \ref{charpsdo} we have that the Gabor matrix $k(w,z)$ of $\sigma(x,D)$ satisfies
 $$|k(w,z)|\lesssim \la z-w\ra ^{-s},\quad w,z\in\rdd,
 $$
 so that $\sigma(x,D)\in FIO(\chi,s)$ with $\chi={\text{Id}}$. The arguments of the proof of the Theorem \ref{5.3} then apply with $w_0=z_0$. The proof of \eqref{5.11eq} is similar.
 \end{proof}
 \par
 We end the paper with some examples of Schr\"odinger
 equations.\par
 Addressing first to non-expert readers, we present some properties of $WF_G (f)$ and treat in this frame the free particle and the harmonic oscillator with smooth potentials, cf. Examples $1,2,3$. The conclusive Example $4$ concerns non-smooth potentials. 
 \begin{proposition}\label{5.5}
Let $f\in\cS'(\rd)$. Then\\
(i) $WF_G\,(\pi(z_0) f)=WF_G (f) $ for every $z_0=(x_0,\xi_0)\in\rdd$.\\
(ii) Let $\delta_{x_0}$ be the Dirac distribution at the point $x_0\in\rd$. Then $$WF_G\,(\delta_{x_0})=\{z=(x,\xi),\, x=0,\xi\not=0\}$$ independently of $x_0$.\\
(iii) Let $\xi_0$ be fixed in $\rd$. Then $$WF_G\, (e^{2\pi i \la x,\xi_0\ra })=\{z=(x,\xi), x\not=0, \xi=0\}$$ independently of $\xi_0$.\\
(iv) Let $c\in\bR$, $c\not=0$, be fixed. Then
 $$WF_G\, (e^{\pi i c|x|^2 })=\{z=(x,\xi), \, x\not=0, \, \xi=c x\}.$$
 \end{proposition}
 \begin{proof}
 The proof of $(i)$ is a consequence of Proposition \ref{5.4}, since $\pi(z_0)=M_{\xi_0}T_{x_0}=\sigma(x,D)$ with $\sigma(x,D)$ being a  pseudodifferential operator with symbol $$\sigma(x,\xi)=e^{2\pi i (\la x, \xi_0\ra-\la x_0,\xi\ra)}\in S^0_{0,0}.$$
Concerning $(ii)$, we are reduced to compute $WF_G\,(\delta)$ since $$WF_G\,(\delta_{x_0})=WF_G\,(T_{x_0}\delta)=WF_G\,(\delta)$$
by   item $(i)$. On the other hand, $V_g (\delta) (x,\xi)=\overline{g(-x)}$. Hence in a small conic neighborhood $\Gamma\subset\rdd$ of the ray $x=t\xi$, $t\in\bR$, $\xi\not=0$, we have rapid decay of $g(-t\xi)$ but for $t=0$, giving the claim.\par
To prove $(iii)$ we proceed similarly as before. From item $(i)$ we obtain that $$WF_G\, (e^{2\pi i \la x,\xi_0\ra })=WF_G \,(M _{\xi_0}1)=WF_G \,(1).$$
On the other hand $|V_g 1(x,\xi)|=|M_{-x}\hat{g}(-\xi)|$ so that
 $|V_g 1(x,\xi)|=|\hat{g}(-\xi)|$ and the arguments of item $(ii)$ give the desired result.\par
 We now prove $(iv)$. We use the Gaussian $g(x)=e^{-\pi |x|^2}$ as a window for the STFT $V_g f$ with $f(x):=e^{\pi i c|x|^2}$. Then standard computations (see also \cite[Theorem 14]{Benyi et all}) give
 $$|V_g f(x,\xi)|=(1+c^2)^{- d /4}e^{-\pi|\xi-c x|^2/(1+c^2)}.
 $$
 The right-hand side is rapidly decaying in any open cone of $\rdd$ excluding the line $\xi-c x=0$. This concludes the proof of the proposition.
 \end{proof}

 \textit{Example 1.}\label{ex1} {\bf The free particle}.\par
 Consider the Cauchy problem for the
Schr\"odinger equation
\begin{equation}\label{cp}
\begin{cases}
i\partial_t u+\Delta u=0\\
u(0,x)=u_0(x),
\end{cases}
\end{equation}
with $x\in\R^d$, $d\geq1$. The  explicit formula for the solution in terms of the kernel is
\begin{equation}\label{sol}
u(t,x)=(K_t\ast u_0)(x),
\end{equation}
where
\begin{equation}\label{chirp0}
K_t(x)=\frac{1}{(4\pi i t)^{d/2}}e^{i|x|^2/(4t)}.
\end{equation}
whereas in terms of classical FIO:
\begin{equation}\label{5.14eq}
u(t,x)= \intrd e^{2\pi i( \la x,\eta\ra -2\pi t
|\eta|^2)} {\widehat {u_0}}(\eta)d\eta.
\end{equation}
The Gabor  matrix  with window function $g(x)=e^{-\pi|x|^2}$ can be controlled (see \cite[Theorem 5.3]{fio-Gelfan-Shilov} even for more general operators):
\begin{equation}\label{5.15eq}
|k(w,z)|\leq C e^{-\eps |z-\chi_t(w)|^2},
\end{equation}
for suitable constants $C>0$ and $\epsilon>0$ and where, for $w=(y,\eta)$,
\begin{equation}
\label{5.16eq} (x,\xi)=\chi_t(y,\eta)=(y+4\pi t \eta,\eta).
\end{equation}
Beside the effectiveness in numerical analysis, cf.\ \cite[Section 6.1]{fio3}, this expression emphasizes the microlocal properties of the propagator. Let us test the propagator of the Gabor wave front set on some particular initial data. If $u_0=\delta$ then $u(t,x)=K_t(x)$ by \eqref{sol}. This is coherent with \eqref{WT1.3} and \eqref{5.16eq}, since from Proposition \ref{5.5}, $(iv)$ and $(ii)$, we have
\begin{align*}
WF_G \,(u(t,x))&= WF_G \,(K_t)=\{(x,\xi),\,x=4\pi t\xi,\,\xi\not=0\}\\
&=\chi_t(WF_G (\delta))=\chi_t(\{(y,\eta),\,y=0,\eta\not=0\})
\end{align*}
 We remark a similar propagation for the initial datum
 $$u_0=K_{-1}(t)=(-4\pi i)^{-d/2} e^{-i |x|^2/4}$$
 for which we have $u_{t=1}=\delta$. Instead, for $u_0=e^{2\pi i \la x,\xi_0\ra}$, with $\xi_0\in\rd$, we have
 $$u(t,x)=e^{- 4 \pi^2 i t |\xi_0|^2} e^{2\pi i \la x,\xi_0\ra}
 $$
and in this case the Gabor wave front set is stuck:
$$ WF_G \,(u(t,x))=WF_G\, (u_0)=\{ (x,0),\,x\not=0\},
$$
by Proposition \ref{5.5}, $(iii)$ and \eqref{5.16eq}.
\vskip0.3truecm

\textit{Example 2.}\label{ex2} {\bf The harmonic oscillator}.\par
Consider the Cauchy problem
\begin{equation}\label{5.17eq}
\begin{cases} i \partial_t
 u -\frac{1}{4\pi}\Delta u+\pi|x|^2 u=0\\
u(0,x)=u_0(x).
\end{cases}
\end{equation}
The solution in terms of a FIO type \eqref{1.13} is
\begin{equation}\label{5.18eq}
u(t,x)=(\cos t)^{-d/2}\intrd e^{2\pi i [\frac1{\cos t} x\eta + \frac{\tan t}2 (x^2+\eta^2)]}\hat{u_0}(\eta)\,d\eta,\quad t\not=\frac\pi 2 + k\pi,\,\,k\in\bZ.
\end{equation}
The Gabor matrix with Gaussian window $g(x)=e^{-\pi|x|^2}$ can be explicitly computed as
\begin{equation}\label{5.19eq}
|k(w,z)|=2^{-\frac d 2} e^{-\frac \pi 2 |z-\chi_t(w)|^2},
\end{equation}
where the canonical transformation is
defined in \eqref{5.20eq}.
Observe that the expression \eqref{5.19eq} is meaningful for every $t\in\bR$. Let us refer to \cite[Section 6.2]{fio3} for applications to numerical experiments. \par
We may test \eqref{WT1.3} on the initial datum $u_0(x)=1$, giving for $t<\pi/2$,
$$u(t,x)=(\cos t)^{-d/2} e^{\pi i \tan t |x|^2}.
$$
From Proposition \ref{5.5}, $(iii)$ and $(iv)$, we have coherently with \eqref{5.20eq}
\begin{align*}WF_G \,(u(t,x))&=\{ (x,\xi),\,x=(\cos t)y, \,\xi=(\sin t) y, y\not=0\}\\
&=\chi_t(WF_G\,(1))=\chi_t (\{(y,\eta),\, y\not=0,\eta=0\}).
\end{align*}
\vskip0.3truecm

\textit{Example 3.} {\bf Smooth potentials}.\par
We now consider the presence in Example $1$ of a potential with symbol in the class $S^0_{0,0}$.  Consider the case
\begin{equation}\label{5.21eq}
e^{-2\pi i \la x_0,\xi\ra },\quad x_0\in\rd\,\mbox{fixed}.
\end{equation}
The related pseudodifferential operator $\sigma(D)$ is the translation operator
\begin{equation}\label{5.22eq}
 \sigma(D)f(x)=T_{x_0}f(x)=f(x-x_0),
\end{equation}
which does not preserve the singular support. Consider first the equation
\begin{equation}\label{5.23eq}
\begin{cases} i \partial_t
 u + \sigma(D) u=0\\
u(0,x)=u_0(x).
\end{cases}
\end{equation}
The solution is given by
\begin{equation}\label{5.24eq}
u(t,x)=e^{i t T_{x_0}} u_0(x)=\intrd e^{2\pi i \la x,\xi\ra } \exp(i t e^{-2\pi i \la x_0, \xi\ra })\,\hat{u_0}(\xi)\,d\xi.
\end{equation}
Despite the nasty oscillations, the symbol of the solution operator belongs to $S^0_{0,0}$ and from Proposition \ref{2}{5.4} we have for every fixed $t\in\bR$,
$$WF_G (e^{i t T_{x_0}} u_0)=WF_G (u_0),
$$
the identity being granted by the fact that $T^{-1}_{x_0}=T_{-x_0}$. Note that the singular support can be expanded. In fact, taking $u_0=\delta$ we have
$$e^{i t T_{x_0}} \delta=\sum_{n=0}^\infty \frac{(it)^n}{n!} \delta_{n x_0}\in\cS'(\rd)
$$
so that sing supp $e^{i t T_{x_0}} \delta= \{n x_0\}_{n\in\bZ_+}$ as soon as $t\not=0$, whereas
$$WF_G \,(e^{i t T_{x_0}} \delta)=WF_G\, (\delta)=\{(0,\xi),\,\xi\not=0\}.
$$
Adding now the potential $\sigma(D)$ to the free particle in Example $1$, we have the Schr\"odinger equation with space-delay
\begin{equation}\label{5.25eq}
\begin{cases}
i\partial_t u+\Delta u+T_{x_0}u=0\\
u(0,x)=u_0(x).
\end{cases}
\end{equation}
Since the operators $e^{it\Delta}$ and $T_{x_0}$ commute, the arguments of of Section $4$ provide as propagator $e^{i t T_{x_0}}e^{it\Delta}$,
that is the convolution with
$$\sum_{n=0}^\infty \frac{(it)^n}{n!}K_t(x-n x_0)\in\cS'(\rd),
$$
where $K_t$ is defined in \eqref{chirp0}. The Gabor propagation is the same as in Example $1$.\par
From a physical point of view, it is perhaps most natural to consider the case when the potential depends on $x$ alone, for example
\begin{equation}\label{5.26eq}
\begin{cases}
i\partial_t u+\Delta u+ M_{\xi_0}u=0\\
u(0,x)=u_0(x),
\end{cases}
\end{equation}
with $M_{\xi_0}u_0=e^{2\pi i \la x,\xi_0\ra }u_0$, $\xi_0$ fixed in $\rd$. Notice that now the operators $e^{i t \Delta}$ and $M_{\xi_0}$ do not commute and, proceeding as in Section $4$ with the perturbation $Bu=M_{\xi_0}u$, we have first to consider
$$B(t)=e^{-i t \Delta} e^{2\pi i \la x,\xi_0\ra } e^{i t \Delta}.
$$
 Omitting further explicit computations, we obtain
\begin{equation}\label{5.27eq}
B(t)=e^{4 \pi^2 i \xi_0^2 t} M_{\xi_0}T_{-4\pi t \xi_0}.
\end{equation}
In principle, one could then continue the computation of the pseudodifferential operator $Q(t)$ in \eqref{eq:kh8} explicitly, and the solution operator will be $e^{it\Delta}Q(t)$. \par 
Observe in \eqref{5.27eq} the presence of the translation factor $T_{4\pi t \xi_0}$, providing same phenomena as before.
\vskip0.3truecm

\textit{Example 4.}\label{ex4} {\bf Non-smooth potentials}.\par
As examples of admissible non-smooth potentials, consider first a non-polynomial homogeneous function $h(z)$, $z=(x,\xi)$, $h(\lambda z)=\lambda^r h(z)$ for $z\not=0$, $\lambda>0$, $r>0$, with $h\in\cC^\infty(\rdd\setminus\{0\})$, and take then as potential any function $\sigma(z)=h(z)$, for $|z|\leq 1$, and $h(z)\in S^0_{0,0}$ for $|z|\geq 1$. This potential satisfies $\sigma\in M^{\infty}_{1\otimes v_{r+ 2d}}(\rdd)$. In fact, we may limit the analysis to the singularity at the origin. From Proposition \ref{fabio3} we have, for $\psi \in\cS(\rdd)$,
\begin{equation}\label{5.28eq}
|V_\psi \sigma (z,\zeta)|\leq C \la \zeta\ra^{-r-2d},\quad z,\zeta\in\rdd.
\end{equation}
We may now return to the discussion about the smoothness  at the origin of the Hamiltonian $a(z)$ in the Introduction. Consider $h(z)$ real-valued non-polynomial homogeneous of degree $2$, $h\in\cC^{\infty}(\rdd\setminus\{0\})$, just to give an example
$$ h(x,\xi)=(|x|^4+|\xi|^4)^{1/2}.
$$
We can include in our analysis the equation
\begin{equation}\label{5.29eq}
\begin{cases}
i\partial_t u+h(x,D)u=0\\
u(0,x)=u_0(x),
\end{cases}
\end{equation}
by absorbing the singularity at the origin into the potential. Namely, take $\f\in\cC^\infty_0(\rd)$, $0\leq\f(z)\leq 1$, $\f(z)=1$ for $|z|\leq 1$, $\f(z)=0$ for $|z|\leq 2$, and split
$$h(z)=a(z)+\sigma(z),\quad a(z)=(1-\f(z))h(z),\,\,\sigma(z)=\f(z)h(z).
$$
At this moment $a(z)$ satisfies the assumptions in the Introduction and the potential $\sigma$ belongs to $M^\infty_{1\otimes v_{2+2d}}(\rdd)$,   in view of \eqref{5.28eq}. We may then apply Theorem \ref{T1.5} to the Cauchy problem \eqref{5.29eq}. Note that the result of propagation should be limited to $u_0\in M^p_{ v_{-r}}(\rd)$ and $\wpr (e^{i t H} u_0)$ with $0<r<1$.\par
Finally, we present an example of non-smooth potential depending on $x$ alone, namely in dimension $d=1$
\begin{equation}\label{5.30eq}
\sigma(x,\xi)=|\sin x|^\mu,\quad \mu>1,\ x,\xi\in\bR.
\end{equation}
By Corollary \ref{fabio2}, $\sigma\in M^\infty_{1\otimes v_{\mu+1}}(\R^2)$. So, consider for instance the perturbed harmonic oscillator in \eqref{5.31eq}.
From Theorem \ref{T1.5} we have that the Cauchy problem is well-posed for $u_0\in M^p_{v_{r}}(\R)$, $|r|<\mu-2$ and the propagation of $\wpr\, (u(t,\cdot))$ for $t\in \bR$ takes place as in Example $2$ for $0<r<\mu/2-1$.
\section*{Appendix}
\begin{proof}[Proof of Theorem \ref{caraI}]
First we prove  {\rm (ii)} $\Rightarrow$  {\rm (i)}. Assume $I\in FIO(\chi_t,s)$ and
\begin{equation}\label{first}
|\langle I\pi(w) g,\pi(z))g\rangle|\leq C(t)\langle z-\chi_t(w)\ra,
\end{equation}
with $C(t)$ positive continuous function on  $]-T, -T[$.
Setting $w=(x,\eta)$ and $z=(x',\eta')$,
using the fact that each component of the mapping $\chi_t(y,\eta)$ and its inverse  is in $\cC^\infty(]-T,T[,\Gamma^1(\rdd))$
we can control the Lipschitz constants of $\chi_t$ and $\chi_t^{-1}$ by continuous constants of $t$
so that the equivalence of \cite[Lemma 4.2]{Wiener}
becomes
\begin{equation}\label{3mezzo}
|\nabla_x\Phi(t,x',\eta)-\eta'|+|\nabla_\eta\Phi(x',\eta)-x|
\asymp_t |\chi _1(t,x,\eta)-x'|+| \chi _2(t,x,\eta)-\eta'|
\end{equation}
for every $x,x',\eta, \eta' \in \rd$
and the implicit constants in the equivalence $\asymp_t $ are continuous with respect to $t\in]-T,T[$.
This reduces the study  to showing that if the operator $I(\sigma_t,\Phi(t,\cdot))$, with $\Phi(t, )$ being the phase related to $\chi_t$ in  \eqref{rel-chi-phi} and satisfying \eqref{eiconal}, fulfils the estimate
\begin{equation}\label{uno}
|\langle I(\sigma_t,\Phi(t,\cdot)) \pi(x,\eta) g,\pi(x',\eta')g\rangle|\leq C(t)\langle \nabla_x
\Phi(t,x',\eta)-\eta',\nabla_\eta \Phi(t,x',\eta)-x\rangle^{-s}
\end{equation}
with $x,x',\eta,\eta'\in\R^d $, $t\in ]-T,T[$, then
\begin{equation}\label{stimar}
 \|\sigma_t\|_{M^{\infty}_{1\otimes v_s}}\leq  C(t),\quad t\in ]-T,T[.
\end{equation}
For $z,w\in\rdd$, let  $\Phi _{2,z}(t,\cdot)$ be the remainder  in the second
 order Taylor expansion  of the phase
 $\Phi(t,\cdot) $, i.e.,
\[
\Phi_{2,z}(t,w)=2\sum_{|\alpha|=2}\int_0^1
(1-\tau)\partial^\alpha\Phi(t,z+\tau w)d\tau \frac{w^\alpha}{\alpha!} \, .
\]
For  a given window $g\in\cS(\rd)$, we set
\begin{equation}\label{psi}
\Psi_z(t,w)=e^{2\pi i \Phi_{2,z}(t,w)}
\big(\overline{g}\otimes\widehat{g}\big)(w).
\end{equation}
Then, the fundamental relation between the Gabor matrix of a FIO and the STFT of its symbol  from \cite[Prop. 3.2]{fio5} and
\cite[Section 6]{fio1} can be rephrased in this framework as
\[
|\langle I \pi(x,\eta)
g,\pi(x',\eta')g\rangle|=|V_{\Psi_{(x',\eta)}}\sigma_t
((x',\eta),(\eta'-\nabla_x\Phi(t,x',\eta),x-\nabla_\eta
\Phi(t,x',\eta)))| \, .
\]
Writing $u=(x',\eta)$, $v=(\eta',x)$, \eqref{uno}  translates into
\[
|V_{\Psi_u(t,\cdot)}\sigma_t(u,v-\nabla\Phi(t,u))|\leq C(t)\langle v-\nabla \Phi(t,u)\rangle^{-s},
\]
and then  into the estimate
\begin{equation}\label{due}
\sup_{(u,w)\in\R^{2d}\times \R^{2d}}\langle w\rangle ^s |V_{\Psi_u(t,\cdot)}\sigma_t(u,w)|\leq C(t).
\end{equation}
The main technical work done in \cite{locNC09} for the time independent case
$\Psi_u(t,\cdot)=\Psi_u(\cdot)$ is to show that the set of windows $\Psi _u$
possesses a joint time-frequency envelope. This property  allows  to write
  $\sigma \in M^{\infty}_{1\otimes v_s}(\R^{2d})$ if and only if
$\sup_{u\in\rdd}|V_{\Psi_u} \sigma|\in L^{\infty}_{1\otimes
v_s}(\R^{4d})$
with \begin{equation}\label{symt}
\|\sigma\|_{ M^{\infty}_{1\otimes v_s}}\asymp
\|\sup_{u\in\rdd}|V_{\Psi_u} \sigma|\|_{L^{\infty}_{1\otimes
v_s}}.
\end{equation}
 The proof of the previous equivalence passes through several lemmas. We point out that the crucial element of the equivalence  is a control of $ |V_\Psi e^{2\pi i \Phi_{2,z}}(u,w)| $, with $\Psi\in\cS(\rdd)\setminus\{0\}$ fixed,  by a
polynomial  $p_\alpha(\partial \Phi_{2,z}(\zeta))$   of derivatives of
   $\Phi_{2,z}$ of degree at most $|\a|$ times a factor that does not depend on $t$. Since $\Phi(t,)\in \cC^\infty(]-T,T[,\Gamma^2(\rdd)) $, we can control the polynomial by a continuous function of $t$ and in the end obtaining that the equivalence \eqref{symt} depends continuously on $t$, which together with \eqref{due} gives \eqref{stimar}. \par

  {\rm (i)} $\Rightarrow$  {\rm (ii)}. If  $I=I(\sigma_t,\Phi_\chi)$ is a FIO of type
   I for $\Phi(t,\cdot)$ and $\chi_t$ in \eqref{rel-chi-phi} and some $\sigma_t\in M^{\infty}_{1\otimes v_s}(\rdd)$ which satisfies \eqref{controllosigma}, then essentially  reading  backwards the  arguments above give  $I(\sigma_t,\Phi_\chi)\in FIO(\chi_t,s)$ with $C(t)$ being a continuous function of $t$.
\end{proof}

\section*{Acknowledgements}
We thank Prof.\ A.\ Vasy for pointing out some references and for useful  comments.


\begin{thebibliography}{10}
\bibitem{wiener1}
K.~Asada and D.~Fujiwara.
\newblock On some oscillatory integral transformations in {$L^{2}({\bf
 R}^{n})$}.
\newblock {\em Japan. J. Math. (N.S.)}, 4(2):299--361, 1978.

\bibitem{Benyi et all}
{\'{A}. B\'{e}nyi, K. Gr\"{o}chenig, K. Okoudjou and L.G. Rogers}.
{ Unimodular Fourier multipliers for modulation spaces}.
{\it J. Funct. Anal.}, 246:366--384, 2007.

\bibitem{wiener3}
J. Bony.
\newblock Op\'erateurs int\'egraux de {F}ourier et calcul de
             {W}eyl-{H}\"ormander (cas d'une m\'etrique symplectique),
{\it Journ\'ees ``\'{E}quations aux {D}\'eriv\'ees {P}artielles''
             ({S}aint-{J}ean-de-{M}onts, 1994)},
             \'Ecole Polytech., Palaiseau, 1--14, 1994.


\bibitem{wiener4}
J. Bony.
\newblock  Evolution equations and generalized {F}ourier integral
             operators,
{\it Advances in phase space analysis of partial differential
             equations}, Progr. Nonlinear Differential Equations Appl., 78, 59--72, Birkh\"auser Boston Inc., Boston, MA, 2009.







\bibitem{Cappiello-shulz} M. Cappiello and R. Shulz.
\newblock Microlocal analysis of quasianalytic Gelfand-Shilov type ultradistributions. arXiv:1309.4236.


\bibitem{locNC09} E.~Cordero, K.~Gr\"ochenig and  F. Nicola.
\newblock Approximation of Fourier integral operators by Gabor multipliers.
\newblock {\em J. Fourier Anal. Appl.}, 18(4):661--684, 2012.

\bibitem{Wiener} E.~Cordero, K.~Gr\"ochenig, F. Nicola and L. Rodino.
\newblock Wiener algebras of Fourier integral operators.
\newblock {\em J. Math. Pures Appl.}, 99:219--233, 2013.

\bibitem{CGNR13} E.~Cordero, K.~Gr\"ochenig, F. Nicola and L. Rodino.
\newblock Generalized metaplectic operators and the Schr\"odinger
equation with a potential in the Sj{\"o}strand class. {\it J. Math. Physics}, to appear. ArXiv:1306.5301.

\bibitem{cn}
E.~Cordero and F. Nicola.
\newblock {Metaplectic representation
on Wiener amalgam spaces and applications to the
 Schr\"odinger
equation}.
\newblock {\em J. Funct. Anal.}, 254:
506-534, 2008.

\bibitem{fio5}
E.~Cordero and F. Nicola.
\newblock Boundedness of Schr\"odinger type propagators on modulation spaces. {\it J. Fourier Anal. Appl.}, 16(3):311--339, 2010.

\bibitem{fio1}
E.~Cordero, F. Nicola and L. Rodino. Time-frequency
analysis of Fourier integral operators. {\it Commun. Pure
Appl. Anal}., 9(1):1--21, 2010.

\bibitem{fio3}
E.~Cordero, F. Nicola and L. Rodino.
\newblock Sparsity of  Gabor representation of Schr\"odinger propagators.
\newblock {\em Appl. Comput. Harmon. Anal.}, 26(3):357--370, 2009.

\bibitem{bertinoro} E. Cordero, F. Nicola and L. Rodino. Schr\"odinger equations in modulation spaces.
Chapter 5 in {\em Studies in Phase Space Analysis with Applications to PDEs}, M. Cicognani, F. Colombini, D. Del Santo Editors, Progress in Nonlinear Differential  Equations and Their Applications, Birkhauser, Basel, 84:81--99, 2013.

\bibitem{fio-evolution}
E.~Cordero, F. Nicola and L. Rodino. \newblock Gabor
representations of evolution operators.  {\it Trans. Amer. Math. Soc.}, to appear. arXiv:1209.0945.

\bibitem{fio-Gelfan-Shilov}
E.~Cordero, F. Nicola and L. Rodino. \newblock Exponentially sparse representations of Fourier integral operators. {\it Rev. Mat. Iberoamer.}, to appear. {\it arXiv:1301.1599}.

\bibitem{CNEdcds} E.~Cordero, F. Nicola and L. Rodino. \newblock Schr\"odinger equations with rough Hamiltonians. {\it Discrete and Continuous Dynamical Systems - Series A}, to appear. arXiv:1312.7791.

\bibitem{Craig-Kappler}
W. Craig, T. Kappeler, W. Strauss.
\newblock Microlocal dispersive smoothing for the Schr\"odinger
equation. {\em Comm. Pure Appl. Math}, 48:769--860, 1995.
\bibitem{Daubechies90}
I. Daubechies.
\newblock The wavelet transform, time-frequency localization and signal analysis.
\newblock {\em IEEE Trans. Inform. Theory}, 36(5):961--1005, 1990.


\bibitem{EN06}
K.-J. Engel and R.~Nagel.
\newblock {\em A short course on operator semigroups}.
\newblock Universitext. Springer, New York, 2006.

\bibitem{F1}  H.~G.~Feichtinger,
\newblock Modulation spaces on locally
compact abelian groups,
\newblock {\em Technical Report, University Vienna, 1983,} and also in
\newblock {\em Wavelets and Their Applications},
M. Krishna, R. Radha,  S. Thangavelu, editors,
\newblock Allied Publishers,  2003, 99--140.
\bibitem{fg89jfa}
H.~G. Feichtinger and K.~Gr{\"o}chenig. Banach spaces
related to integrable group representations and their
atomic decompositions. I. {\it J. Funct. Anal}., 86(2):307--340, 1989.

\bibitem{folland89}
G.~B. Folland.
\newblock {\em Harmonic analysis in phase space}.
\newblock Princeton Univ. Press, Princeton, NJ, 1989.

\bibitem{gz} S. Graffi, L. Zanelli.
\newblock Geometric approach to the Hamilton-Jacobi equation and global
parametrices for the Schr\"odinger propagator. {\it Reviews in
Mathematical Physics}, 23:969--1008, 2011.

\bibitem{charly06}
K.~Gr{\"o}chenig. Time-Frequency Analysis of Sj{\"o}strand's
Class. {\it Rev. Mat. Iberoamericana}, 22(2):703--724, 2006.
\bibitem{book}
K.~Gr{\"o}chenig. {\it Foundations of time-frequency
analysis}. Applied and Numerical Harmonic Analysis.
Birkh\"auser Boston, Inc., Boston, MA, 2001.





\bibitem{GR} K.~Gr{\"o}chenig and Z. Rzeszotnik. Banach algebras of
 pseudodifferential operators and their almost diagonalization.
 Ann. Inst. Fourier, 58(7):2279-2314, 2008.

%
\bibitem{Hassel-Wunsch}
A. Hassell and J. Wunsch.
\newblock The Schr\"odinger propagator for scattering metrics. {\em Ann. of Math.}, 162:487--523, 2005.
\bibitem{helffer84}
B. ~Helffer.
\newblock { T}h\'eorie Spectrale pour des Operateurs
 Globalement Elliptiques.
\newblock {\em Ast\'erisque}, Soci\'et\'e Math\'ematique
de France, 1984.
\bibitem{hormander3}
 L.~H\"{o}rmander.
\newblock {\it The Analysis of Linear Partial Differential
Operators}, Vol. I and Vol. III, Springer-Verlag, 1985.

\bibitem{hormanderglobalwfs91}
L.~H\"{o}rmander. \ {Quadratic hyperbolic operators}, in {\em
``Microlocal analysis and applications''}, 118--160, Lecture Notes
in Math., 1495, Springer, Berlin, 1991.

\bibitem{ito} K. Ito. Propagation of Singularities for 
Schr\"odinger Equations on the Euclidean 
Space with a Scattering Metric. {\it Comm. Partial Differential Equations}, 31:1735--1777, 2006. 
\bibitem{ito-nakamura} K. Ito and S.  Nakamura.  Singularities of solutions to Schr\"odinger equation on scattering manifold. {\it Amer. J. Math.}, 131(6):1835--1865, 2009.

\bibitem{js94} A. Jensen and S. Nakamura. Mapping properties of functions of Scr\"odinger operators between $L^p$-spaces and Besov spaces. {\it Advanced Studies in Pure Mathematics, Spectral and Scattering Theory and Applications}, {23}:187--209, 1994.\par

\bibitem{js95} A. Jensen and S. Nakamura. { $L^p$-mapping properties of functions of Schr\"odinger operators and their applications to scattering theory}. {\it J. Math. Soc. Japan}, {47}(2):253--273, 1995.\par

\bibitem{kki1} {K. Kato, M. Kobayashi and S. Ito}.
\ { Representation of Schr\"odinger operator of a free particle
via short time Fourier transform and its applications.} {\it
Tohoku Math. J.}, 64:223--231, 2012.

\bibitem{kki2} {K. Kato, M. Kobayashi and S. Ito}.
\ { Remark on wave front sets of solutions to Schr\"odinger equation
of a free particle and a harmonic oscillator.} {\it SUT J.Math.}, 47:175-183, 2011.


\bibitem{kki4} {K. Kato, M. Kobayashi and S. Ito}.
{\
Estimates on Modulation Spaces for Schr\"odinger
Evolution Operators with Quadratic and Sub-quadratic
Potentials}. arXiv:1212.5710.

\bibitem{Martinez}
{A. Martinez, S. Nakamura and V. Sordoni}.
\ { Analytic smoothing effect for the Schr\"odinger equation with long-range perturbation}.
{\it Comm. Pure Appl. Math.}, 59:1330--1351, 2006.
\bibitem{Martinez2}
{A. Martinez, S. Nakamura and V. Sordoni}.
\ { Analytic wave front set for solutions to Schr\"odinger equations}.
{\it Adv. Math.}, 222(4):1277--1307, 2009.
\bibitem{Melrose} {R. Melrose}.
{\ Spectral and scattering theory for the Laplacian on
asymptotically Euclidean spaces}. In {\em ``Spectral and
Scattering theory''}, Sanda, 1992, in: Lect. Notes Pure Appl.
Math., Dekker, New York, 161:85--130, 1994.

\bibitem{Miyachi-NicolaRivetti-Tabacco-Tomita}
{A. Miyachi, F. Nicola, S. Rivetti, A. Tabacco and
N. Tomita}.
\ { Estimates for unimodular Fourier multipliers on modulation spaces}.
{\it Proc. Amer. Math. Soc.}, 137:3869--3883, 2009.
\bibitem{Mizuhara}
R. Mizuhara.
\newblock Microlocal smoothing effect for the Schr\"odinger evolution  equation in Gevrey classes. {\it J. Math. Pures Appl.}, 91:115-136, 2009.

\bibitem{Nakamura}
S. Nakamura.
\newblock Propagation of the homogeneous wave front set for
Schr\"odinger equations, {\it Duke Math. J.}, 126(2):349--367,
2005.
\bibitem{Nakamura2}
S. Nakamura.
\newblock Semiclassical singularity propagation property for
Schr\"odinger equations. {\it J. Math. Soc. Japan}, 61(1):177--211,
2009.

\bibitem{nicola} F. Nicola. Phase space analysis of semilinear parabolic equations, {\it J. Funct. Anal.}, 267:727--743, 2014. 


\bibitem{RS75}
M.~Reed and B.~Simon.
\newblock {\em Methods of modern mathematical physics. {II}. {F}ourier
  analysis, self-adjointness}.
\newblock Academic Press [Harcourt Brace Jovanovich Publishers], New York,
  1975.


\bibitem{Robbiano-Zuily}
L.~Robbiano and C.~Zuily.
\newblock {Microlocal analytic smoothing effect  for the Schr\"odinger equation}.
{\em Duke Math. J.}, 100:93--129, 1999.

\bibitem{Robbiano-Zuily2}
L.~Robbiano and C.~Zuily.
\newblock {Analytic theory for the quadratic scattering wave front set and applications to the Schr\"odinger equation}.
{\em Ast\'erisque}, 283:1--128, 2002.

\bibitem{RWwavefrontset}{L. Rodino and P. Wahlberg}.
\ { The Gabor wave front set}. arXiv:1207.5628v2.

\bibitem{sw} R. Schulz and P. Wahlberg. The equality of the homogeneous and the Gabor wave front set. arXiv.1304.7608.

\bibitem{Shubin91}
M.~A. Shubin.
\newblock {\em Pseudodifferential Operators and Spectral Theory}.
\newblock Springer-Verlag, Berlin, second edition, 2001.
\newblock Translated from the 1978 Russian original by Stig I. Andersson.
\bibitem{wiener30}   J. Sj\"ostrand. An algebra of pseudodifferential operators. {\em Math. Res. Lett.}, 1(2):185--192, 1994.
\bibitem{wiener31} J. Sj{\"o}strand.
\newblock Wiener type algebras of pseudodifferential operators.
\newblock In {\em S\'eminaire sur les \'Equations aux D\'eriv\'ees Partielles,
 1994--1995}, pages Exp.\ No.\ IV, 21. \'Ecole Polytech., Palaiseau, 1995.


\bibitem{tataru} D. Tataru. Phase space transforms and microlocal analysis.
{\it Phase space analysis of partial differential equations}, Vol. II, 505--524, Pubbl. Cent. Ric. Mat. Ennio Giorgi, Scuola Norm. Sup., Pisa, 2004.\\
http://math.berkeley.edu/\%7Etataru/papers/phasespace.pdf



\bibitem{wh}
B. Wang and H. Hudzik.
{\ The global Cauchy problem for the NLS and NLKG
 with small rough data}.
{\it J. Differential Equations}, 232:36--73, 2007.

\bibitem{Weinstein}
A. Weinstein.
\newblock A symbol class for some Schr\"odinger equations in
$\bR^n$,  {\it Amer. J. Math.}, 107(1):1-21, 1985.

 \bibitem{wunsch} J. Wunsch. Propagation of singularities and growth for Schr\"odinger operators. {\it Duke Math. J.}, 98(1):137--186, 1999. 




\end{thebibliography}
\end{document}